\documentclass[a4paper]{amsart}
\usepackage{enumerate}
\usepackage[shortlabels]{enumitem}
\usepackage[margin=3.6cm]{geometry}
\usepackage{amscd, amssymb, pb-diagram}
\usepackage{tikz-cd}
\usepackage{amsthm} 
\usepackage[nospace,noadjust]{cite}
\usepackage{lmodern}
\usepackage{graphicx}
\usepackage[toc,page]{appendix}
\usepackage[normalem]{ulem}
\usepackage{hyperref}
\usepackage{tikz}
\usepackage{csquotes}
\usepackage{color}
\usepackage{a4wide}

\title[Quantum isometry groups of Cuntz--Krieger algebras]{Quantum isometry groups of log-Laplacians on Cuntz--Krieger algebras}

\author{Amaury Freslon}
\address{Universit{\'e} Paris-Saclay, CNRS, Laboratoire de Math{\'e}matiques d’Orsay, 91405 Orsay, France}
\email{amaury.freslon@universite-paris-saclay.fr}

\author{Dimitris M. Gerontogiannis}
\address{Institute of Mathematics of the Polish Academy of Sciences, ul. {\'S}niadeckich 8, 00–656, Warszawa, Poland}
\email{dgerontogiannis@impan.pl}

\author{Adam Skalski}
\address{Institute of Mathematics of the Polish Academy of Sciences, ul. {\'S}niadeckich 8, 00–656, Warszawa, Poland}
\email{a.skalski@impan.pl}


\DeclareMathOperator{\Aut}{Aut}
\DeclareMathOperator{\Dom}{Dom}
\DeclareMathOperator{\Fix}{Fix}
\DeclareMathOperator{\Pol}{Pol}
\DeclareMathOperator{\QAut}{QAut}

\DeclareMathOperator{\id}{id}
\DeclareMathOperator{\supp}{supp}

\newcommand{\CA}{C(\QG_{A}^{\ell})}
\newcommand{\CAi}{C(\QG_{A}^{\infty})}
\newcommand{\CG}{C(\QG)}
\newcommand{\ot}{\otimes}
\newcommand{\F}{\mathbb{F}}
\newcommand{\GA}{\QG_{A}^{\ell}}
\newcommand{\GAi}{\QG_{A}^{\infty}}
\newcommand{\LC}{C_{c}^{\infty}(\Gamma_{A})}
\newcommand{\N}{\mathbb{N}}
\newcommand{\QG}{\mathbb{G}}
\newcommand{\QH}{\mathbb{H}}
\newcommand{\T}{\mathbb{T}}
\newcommand{\Z}{\mathbb{Z}}
\newcommand{\QGc}{\mathbb{G}_{\textup{clas}}}

\def\d{\operatorname{d}}

\def\df{\operatorname{d_f}}
\def\op{\operatorname{op}}
\def\spec{\operatorname{sp}}

\numberwithin{equation}{section}

\theoremstyle{theorem}
\newtheorem{thm}{Theorem}[section]
\newtheorem{cor}[thm]{Corollary}

\newtheorem{lemma}[thm]{Lemma}
\newtheorem{prop}[thm]{Proposition}

\newtheorem{thmx}{Theorem}

\theoremstyle{definition}
\newtheorem{definition}[thm]{Definition}

\theoremstyle{remark}
\newtheorem{remark}[thm]{Remark}

\newtheorem{quest}[thm]{Question}
\newtheorem*{Acknowledgements}{Acknowledgements}

\makeatletter
\@namedef{subjclassname@2020}{%
  \textup{2020} Mathematics Subject Classification}
\makeatother

\begin{document}

\begin{abstract}
We compute the quantum isometry groups of Cuntz--Krieger algebras endowed with the spectral triples coming from the Ahlfors regular structure of the underlying topological Markov chain. This allows us to exhibit a new family of compact quantum groups, mixing features from quantum automorphism groups of graphs and easy quantum groups. Contrary to the classical isometry groups, whose actions on the Cuntz--Krieger algebras are never ergodic, the quantum isometry group acts ergodically in the case of the Cuntz algebra.
This also leads to the construction of a (genuinely quantum) ergodic action of a compact matrix quantum group on the Cantor space.
\end{abstract}



\maketitle

\section{Introduction}
A classical result in differential geometry asserts that the isometry group of a compact Riemannian manifold is a compact Lie group. This group controls rigidity, classification phenomena (\cite[Chapter II]{Kob}), and relates to spectral data of the Laplace--Beltrami operator. Over the last few decades, this picture has expanded in two directions at once. Riemannian manifolds are replaced by $C^*$-algebras equipped with a spectral triple, and isometry groups by compact quantum groups dictated by spectral data.  Spectral triples are the quintessence in Connes' noncommutative geometry \cite{connes94noncommutative}, equipping $C^*$-algebras with features pertinent to Riemannian geometry via Dirac operators.
Compact quantum groups, which have appeared first as deformations of classical Lie groups (\cite{woronowicz}), were later understood as  structures encoding \emph{quantum} symmetries of $C^*$-algebras, following Wang's seminal work \cite{wang98quantum} -- see \cite{Adam} and references therein.

It then becomes a meaningful task to investigate quantum symmetries of a $C^*$-algebra which preserve a given spectral triple, and in that sense are quantum isometries. This line of thought was initiated by Goswami in \cite{goswami09}, and continued by several other authors (see \cite{bhowmick16quantum}). Very roughly speaking, given a $C^*$-algebra $A$ equipped with a spectral triple $(\mathcal{A},H,D)$ where $\mathcal{A}\subset A$ is dense, one can consider the category of compact quantum groups that act on $A$ so that the action preserves the domain of $D$ and commutes with it. Such actions can be termed as $D$-\textit{isometric}, and the universal object of this category (if it exists) is called the \emph{quantum isometry group}.

Here we explore the quantum isometries of a Cuntz--Krieger algebra $O_A$ of a finite directed graph with adjacency matrix $A$. These are highly noncommutative $C^*$-algebras encoding aspects of the orbit space of topological Markov chains.
Despite their fundamental role since the 1980s, it was only recently that a canonical Dirac-type operator $D$ was associated with them  in \cite{gerontogiannis25heat}. The key tool was the so-called log-Laplacian introduced in \cite{gerontogiannis25dirichlet}, a non-local diffusion operator on the underlying Ahlfors regular shift space.  The latter operator is defined on any metric space equipped with a finite \emph{Ahlfors regular measure}. In the context of Cuntz--Krieger algebras, the construction in \cite{gerontogiannis25heat} proceeds by endowing the associated Deaconu--Renault groupoid $\Gamma_A$ with an Ahlfors regular structure and then, via additional intermediate steps, defining a Hamiltonian-type operator $D:=-\Delta + V$ which yields a spectral triple on $O_A$. The relevant  Hilbert space $L^2(\Gamma_A)$ can be identified with the GNS space $L^2(O_A,\tau)$, where $\tau$ is the unique KMS state of the $\mathbb T$-action on $O_A$. The classical isometry group associated with $D$, defined following the notion introduced in \cite{Park95Isometries}, was shown to be the compact Lie group $\mathbb{T}\wr \Aut(A)$, where $\Aut(A)$ is the automorphism group of $A$.

The computation of the \emph{quantum} isometry group in fact leads to a whole new family of compact quantum groups. Further, we should stress that so far the concept of quantum symmetry groups of graph C$^*$-algebras studied in the literature (\cite{JoardarMandal}, \cite{SchmidtWeber}) always assumed the action to be \emph{linear} in terms of the natural generating set. This turns out to be the case also here, but we derive it in a non-trivial way from the underlying analytic assumptions. 
The definition of the quantum groups mentioned above and the main result of our paper are summarised in the following theorem.

\begin{thmx} \label{thmA}
Let $\ell\in \mathbb N \cup \{\infty\}$ and consider the universal $C^*$-algebra $\CA$ generated by partial isometries $u=(u_{\alpha, \beta})_{1\leq \alpha,\beta\leq N}$, with range projections $p=(p_{\alpha, \beta})_{1\leq \alpha,\beta\leq N}$ and source projections $q=(q_{\alpha, \beta})_{1\leq \alpha,\beta\leq N}$, such that
	
	\begin{enumerate}[(I)]
		\item $p,q$ are magic unitaries preserving the right Perron--Frobenius eigenvector $\vec{u}$ of $A$; i.e. $p\vec{u}=\vec{u}1=q\vec{u}$;
		\item $A$ intertwines $p$ and $q$; i.e. $Ap=qA$;
		\item for $n \leq \ell$ and $\alpha, \beta\in V_A^{n}$, the element $u_{\alpha,\beta}:=u_{\alpha_1,\beta_1}\ldots u_{\alpha_n,\beta_n}$ is a partial isometry. 
	\end{enumerate}
	Then, $\CA$ equipped with a natural coproduct, defines a compact matrix quantum group of Kac type. The quantum group $\QG_A^\ell$ acts on $O_A$ and for $\ell=\infty$ the action is $D$-isometric. Finally, any $D$-isometric action factors through the $\GAi$-action on $O_A$. In other words, $\GAi$ is the quantum isometry group of the spectral triple on $O_A$ associated with $D$.
\end{thmx}

We call $\GA$ the $\ell$\emph{-Ariadne quantum group}. It is clear that for $\ell \leq \ell'$, $C(\mathbb G_{A}^{\ell'})$ is a subgroup of $\CA$. Further, they are genuinely quantum as they contain the dual of the free group $\F_{N}$, where $N$ is the number of vertices in the graph (this is responsible for what might be viewed as the \emph{quantum gauge action} on $O_A$). From the spectral standpoint of $D$, the picture of $\GA$ is as follows: view the infinite path space $\Sigma_A$ as the boundary of the rooted tree with root the empty path, and $\ell$-level vertices the finite paths of length $\ell\in \mathbb N$. This rooted tree provides a refining sequence of finite resolutions of $\Sigma_A$. Then, $\GA$ encodes the quantum symmetries of the dynamics in the Cuntz--Krieger algebra that happen at resolution $\ell$, where the latter is related to a finite (controlled by $\ell$) truncation of the spectrum of $D$. 

Therefore, in order to discover the quantum isometry group $\GAi$ one has to go through the whole \textit{thread} of the $\ell$-Ariadne groups $\GA$. Moreover, their behaviour is quite intricate, as seen from the following facts.

\begin{enumerate}
	\item There are instances that the thread does not terminate. A primary example is the full shift were the matrix $A:=\bf{1}_N$ consists only of $1$'s. In that case we prove that each $\QG_{\bf{1}_N}^\ell$ is a unitary easy quantum group in the class discovered by Mang in his thesis \cite{Mang2022phd}. Moreover, the quantum groups $\QG_{\bf{1}_N}^\ell$ are mutually non-isomorphic. 
	\item It can happen that the thread terminates in one step, namely $\QG_A^1=\GAi$. For example when $A$ is the adjacency matrix of the Fibonacci graph, in which case the quantum groups coincide with $\widehat{\F_{N}}$.
	\item The classical version of $\QG_A^1$ (arising from the abelianization of the algebra $C(\QG_A^1)$) coincides with that of $\GAi$, and is the classical isometry group of the spectral triple on $O_A$ associated with $D$, mentioned above. In other words, the classical isometry group can be derived from a finite truncation of $D$, whereas in general, the quantum one requires all of $D$.
\end{enumerate}

Moreover, the quantum groups $\GA$ appear to be closely connected both to \emph{quantum automorphism groups of graphs} (\cite{banica05quantum}), denoted by $\QAut(A)$, and to \emph{easy quantum groups} (\cite{banica09liberation}, \cite{freslon2023compact}). The full relationship between  $\QAut(A)$ and $\QG_A^\infty$ remains somewhat elusive, but we can see for example that $\QG_{\bf{1}_N}^\infty$ admits 
$\mathbb S_N^+\times \mathbb S_N^+$ as a \emph{subquotient}.

This leads us to the next main result of this paper, i.e.\ the fact that, although the action of the classical isometry group on $O_A$ is never ergodic, in the case of the Cuntz algebra the quantum isometry group does act ergodically.

\begin{thmx} \label{thmB}
	The action of $\QG_{\bf{1}_N}^\infty$ on $O_N$ is ergodic.
\end{thmx}

The above theorem tells us that our quantum isometry groups are much larger than their classical counterparts, which should be contrasted with the fact that quantum isometry groups of classical smooth manifolds coincide with their classical versions (\cite{goswamijoardar}). One might consider this difference to arise as a consequence of the fact that our `quantum manifold' $O_A$ is very non-commutative, or from the fact that our arguments also establish an existence of a compact quantum group acting ergodically on the Cantor set. Such quantum actions were considered in \cite{bhowmickgoswamiskalski}, \cite{bassiconti} and \cite{bronwlow25self}, where the constructions involved certain inductive limits. Here however, for the first time, we produce an example of an ergodic action of a compact \emph{matrix} quantum group.

\begin{thmx} \label{thmC}
	The action of $\GAi$ preserves the path space $\Sigma_A$ (i.e.\ a Cantor set). In particular, the (non-classical) compact matrix quantum group $\QG_{\bf{1}_N}^\infty$  acts ergodically on $\Sigma_{\bf{1}_N}$.
\end{thmx}

The detailed plan of the paper is as follows: after this introduction, in Section \ref{sec:preliminaries} we recall the connection between topological Markov chains, Cuntz--Krieger algebras, their groupoid picture and the associated spectral triples. Section \ref{sec:Quantum_isometries} introduces the appropriate isometric actions of compact quantum groups and defines the compact quantum groups $\QG_A^\ell$ which are the main objects of our study. This prepares the ground for Section \ref{sec:main}, where we show our main result, namely that $\QG_A^\infty$ is the isometry group of the Cuntz--Krieger algebra equipped with the relevant spectral triple. This is the content of Theorem \ref{thmA}, derived from Propositions \ref{prop:actions}, \ref{prop:Qgroup} and Theorem \ref{thm:factoring}. Section \ref{sec:properties} discusses the classical versions of our quantum groups, and presents properties of the relevant actions, in particular their ergodicity in the Cuntz algebra case. Here we establish Theorem \ref{thmB} (Theorem \ref{thm:Erg}) and Theorem \ref{thmC} (Theorem \ref{thm:ergfaith}).

\section{Preliminaries}\label{sec:preliminaries}

In this section we present a summary of background information regarding the Cuntz--Krieger algebras (\cite{cuntz80class}), focusing on their groupoid picture and a construction of spectral triples in \cite{gerontogiannis25heat}. All along this work, $N\geq 2$ will be a fixed non-negative integer,  and  $A\in M_N(\{0,1\})$ will be a \emph{primitive} matrix, meaning that there exists $k\in \N$ such that all the entries of $A^{k}$ are strictly positive. We will write $\N_0:= \N \cup \{0\}$.

\subsection{Topological Markov chains and Deaconu--Renault groupoid}

In this section, we will recall the main results of \cite{gerontogiannis25heat}, which lead to a construction of a natural spectral triple on Cuntz--Krieger algebras. The definition of that triple involves the geometry of the underlying topological Markov chain, and is best given in the setting of the corresponding Deaconu--Renault groupoid (\cite{deaconu}, \cite{kumjian97graphs}).

\subsubsection{Topological Markov chains}
The reader can find more on this topic in \cite{Katok_Has}. As mentioned above, $A\in M_N(\{0,1\})$ will be a primitive matrix. A word $x_{1}\ldots x_{k}\in\{1, \ldots, N\}^{k}$ is said to be \emph{admissible} if $A_{x_{i}x_{i+1}} = 1$ for all $1\leqslant i\leqslant k-1$ and for $k\in \N_0$ we denote by $V_A^k$  the collection of admissible words of length $k$, where $V_A^0$ contains only the empty word ${\o}$. We also set 
$$V_A:= \bigsqcup_{k\in \N_0} V_A^k.$$
 Often we will denote the length $k$ of $\alpha\in V_A^k$ by $|\alpha|$. 

Equip $\{1,\ldots , N\}^{\mathbb N}$ with the product topology and consider the closed subset of the latter space, 
$$\Sigma_A := \{x=(x_n)_{n\in \mathbb N} \in \{1,\ldots , N\}^{\mathbb N}: A_{x_n,x_{n+1}}=1\}.$$ 
A basis of open compact sets for the topology on $\Sigma_A$ is given by the \emph{cylinder sets} associated to finite words $\alpha = \alpha_1\ldots \alpha_n \in V_A$, defined as $$C(\alpha):= \{x\in \Sigma_A: x_i = \alpha_i,\,\, \text{for } 1\leq i \leq n\}.$$ The convention is that $C({\o}):= \Sigma_A.$ The dynamics on $\Sigma_A$ is given by the left shift map $\sigma_A: \Sigma_A \to \Sigma_A$, which is a local homeomorphism, giving rise to the topological Markov chain $(\Sigma_A,\sigma_A)$. 

For the metric structure on $(\Sigma_A,\sigma_A)$, let $\lambda >1$ be the \emph{expansion constant} (its precise value will play no role in what follows) and equip $\Sigma_A$ with the ultrametric
$$
\mathrm{d}(x,y):=\lambda^{-\inf \{n-1: n \in \N,x_n\neq y_n\}}, \;\; x, y \in \Sigma_A,
$$
with the convention that $\inf \varnothing = \infty$. We have that $\mathrm{d}(x,y)=\lambda^{-n}$ when there is a word $\alpha \in V_A^n$ such that $x=\alpha x'$ and $y=\alpha y'$, where $x'$ and $y'$ are infinite admissible words starting with different letters. Observe that for every $x\in \Sigma_A, n \in \N$ we have
$$
B(x,\lambda^{-n})=C(x_1\ldots x_n).
$$
Also, the shift map is locally expanding, and for $n>1$ it restricts to a homeomorphism 
$$B(x,\lambda^{-n})\to B(\sigma_A(x),\lambda^{-n+1}).$$

The canonical $\sigma_A$-invariant measure on the topological Markov chain $(\Sigma_A,\sigma_A)$ is the \emph{Parry measure}. It is defined in the following way. By $\lambda_{\max}>1$ we denote the \emph{Perron--Frobenius} (i.e.\ maximal) eigenvalue of $A$, by $\vec{u}=(u_j)_{j=1}^N$ the corresponding eigenvector and $\vec{v}=(v_j)_{j=1}^N$ the \emph{Perron--Frobenius eigenvector} of $A^T$. We first normalise $\vec{u}$ so that 
$$\sum_{j=1}^{N}u_j=1.$$ 
Then $\vec{u}$ itself is a probability distribution. We also normalise $\vec{v}$ so that $\vec{u}\cdot \vec{v} =1$. We denote this distribution by $\vec{p}:=(p_j=u_jv_j)_{j=1}^{N}$. 

Moreover, one can form the stochastic matrix $P\in M_N\{[0,1]\}$,
$$P_{i,j}:=\frac{A_{i,j}u_j}{\lambda_{\max}u_i}, \;\;\; i, j=1, \ldots, N$$
which measures the probability of transitioning from a vertex $i$ to $j$. The distribution $\vec{p}$ is stationary for $P$; i.e. $\vec{p}P=\vec{p}$. This gives rise to the $\sigma_A$-invariant Parry measure $\nu$ defined on cylinder sets $C(\alpha)$ for $\alpha=\alpha_1\ldots \alpha_n \neq {\o}$ by 
$$\nu(C(\alpha)):=p_{\alpha_1}P_{\alpha_1,\alpha_2}\ldots P_{\alpha_{n-1},\alpha_n}=\frac{v_{\alpha_1}u_{\alpha_n}}{\lambda_{\max}^{n-1}}.$$
Moreover, a useful property for us is that there is some $C\geq 1$ so that for every $x\in \Sigma_A$ and $0\leq r \leq 1$,
\begin{equation}\label{eq:Ahlfors_reg}
C^{-1}r^{\df}\leq \nu(B(x,r))\leq Cr^{\df}, \quad \df=\log_{\lambda}(\lambda_{\max}),
\end{equation}
where $\lambda$ is the expansion constant introduced above.
In other words, the measure $\nu$ is \emph{Ahlfors $\df$-regular}. Here $\df$ coincides with the Hausdorff dimension of $(\Sigma_A,d).$ Note that $\log(\lambda_{\max})>0$ is the topological entropy of $(\Sigma_A,\sigma_A)$. Also, following \cite{gerontogiannis25heat}, for every $r>0, s>0$ and $x \in \Sigma_A$ we have
\begin{equation}\label{eq:Ahlfors_int}
\int_{B(x,r)} \frac{1}{\d(x,y)^{\df -s}} \d \mu (y) \simeq r^s.
\end{equation} 

Contrary to \cite{gerontogiannis25heat}, here we will employ the log-Laplacian with respect to the \emph{conformal} version of $\nu$. This produces the same operator up to a canonical unitary intertwiner. Consider the continuous function $h:\Sigma_A\to (0,\infty)$ defined as $h(x):=v_{x_1}$, and define the Borel probability measure $$\d \mu:=\frac{1}{h}\d \nu.$$ On cylinder sets it is given by 
$$\mu(C(\alpha))=\frac{u_{\alpha_n}}{\lambda_{\max}^{n-1}}, \;\;\; \alpha = \alpha_1 \cdots \alpha_n \in V_A^n.$$ 
The interesting part of $\mu$ is that it is $\lambda_{\max}$-conformal, and so for all $\alpha$ as above and $k \in \N_0$, $k\leq |\alpha|$ we have 
\begin{equation}\label{eq:conformal}
\mu(\sigma^{k}(C(\alpha)))=\lambda_{\max}^k\mu(C(\alpha)).
\end{equation}
Also, $\mu$ is an $\lambda_{\max}$-eigenvector of the dual of the \emph{transfer map} $L:C(\Sigma_A)\to C(\Sigma_A)$, defined for $f \in C(\Sigma_A)$ and $x \in \Sigma_A$ by 
 $$(Lf)(x):=\sum_{y\in \sigma_A^{-1}(x)} f(y).$$

\subsubsection{The groupoid picture}

We now recast the topological Markov chain $(\Sigma_{A},\sigma_{A})$ in a groupoid context, so as to make connections with operator algebras and non-commutative geometry. We refer the reader, for instance, to \cite{williams19tool} and \cite{Aidan} for a comprehensive treatment of groupoids and their associated $C^*$-algebras. 

We consider  the {\'e}tale groupoid 
$$
\Gamma_A=\{(x,n,y)\in \Sigma_A\times \mathbb Z \times \Sigma_A: \exists k\in \N_0\,\, \text{such that}\,\,n+k\geq 0 \,\,\text{and}\,\,  \sigma_A^{n+k}(x)=\sigma_A^{k}(y)\}\rightrightarrows \Sigma_A.
$$
The source and range maps of $\Gamma_A$ are given by the formulae
$$
s(x,n,y):=y,\qquad r(x,n,y):=x,
$$
and the partial multiplication and inversion are given by 
$$m((x,n,y),(y,\ell,z))\equiv(x,n,y)(y,\ell,z):=(x,n+\ell,z),\quad (x,n,y)^{-1}:=(y,-n,x).$$
Further, consider the maps $\kappa:\Gamma_A\to \mathbb{N}_0$ and $c:\Gamma_A\to \mathbb Z$ defined as 
\begin{align*}
\kappa(x,n,y)&:=\min \left\{k\geq \max \{0,-n\}: \sigma_A^{n+k}(x)=\sigma_A^k(y)\right\},\\
c(x,n,y)&:=n.
\end{align*}

Note that \cite{gerontogiannis25heat} used the symbol $G_A$ for the groupoid above; we change it here, as to avoid the confusion with the (quantum) isometry groups to be considered later in the text.

We equip $\Gamma_A$ with the coarsest topology that makes the maps $s,r,\kappa,$ and $c$ continuous. Then, $\Gamma_{A}$ becomes a locally compact, totally disconnected \'etale groupoid. The map $c:\Gamma_A\to\Z$ is a groupoid homomorphism and the formula
\begin{equation}
\alpha_{z}(f)(x,n,y):=z^{n}f(x,n,y),\quad f\in C_{c}(\Gamma_{A}), \,\, (x, n,y) \in \Gamma_A, \, z\in\mathbb{T},\end{equation}
defines an action of the circle $\mathbb{T}$ by $*$-automorphisms of the reduced $C^{*}$-algebra $C^{*}_{r}(G_{A})$, called the \textit{gauge action}. A basis for the topology of $\Gamma_{A}$ is given by \emph{bisections} associated to $\alpha$, $\beta\in V_A$;
$$C(\alpha,\beta):=\{(x,|\alpha|-|\beta|,y)\in \Gamma_A: x\in C(\alpha), y\in C(\beta), \sigma_A^{|\alpha|}(x)=\sigma_A^{|\beta|}(y)\}.$$

According to \cite{gerontogiannis25heat}, $\Gamma_A$ admits a decomposition over a finer collection of clopen bisections as well as a metric-measure structure. Namely, 
\begin{equation}\label{eq:{A}decomp}
\Gamma_A=\bigsqcup_{\gamma\in I_A} \Gamma_{\gamma}.
\end{equation} 
Here, the index set is
\begin{equation}\label{eq:indexset}
I_A:=\{\gamma=\alpha.\beta\in V_A\times (V_A\setminus \{{\o}\}):\; \alpha={\o}\;\mbox{or}\; (\alpha\neq {\o}, A_{\alpha_{|\alpha|},\beta_{|\beta|}}=1, \alpha_{|\alpha|}\neq \beta_{|\beta|-1})\}.
\end{equation}
In \eqref{eq:indexset} we make the convention that $\beta_0={\o}.$ The sets appearing in the decomposition are 
\begin{equation*}
\Gamma_{\alpha.\beta}:= \left\{(x,n,y)\in \Gamma_A: \begin{matrix} x=\alpha\sigma_A^{|\beta|-1}(y), \; y=\beta\sigma_A^{|\beta|}(y),\; \mbox{and}\\ 
n=|\alpha|-|\beta|+1, \; \kappa(x,n,y)=|\beta|-1\end{matrix} \right\}.
\end{equation*}
For $\gamma=\alpha.\beta\in I_A$, we write
$$r(\gamma):=\alpha\quad\mbox{and}\quad s(\gamma):=\beta.$$
By construction, we have that 
\begin{equation*}
\kappa|_{\Gamma_\gamma}=|s(\gamma)|-1 \quad\mbox{and}\quad c|_{\Gamma_\gamma}=|r(\gamma)|-|s(\gamma)|+1.
\end{equation*}
Moreover, for each $\gamma=\alpha.\beta \in I_A$ the clopen set $\Gamma_\gamma$ is a subset of $C(\alpha \beta_{|\beta|},\beta)$, and hence a bisection as well. Also, the restricted source map 
$$s_{\gamma}:=s|_{\Gamma_\gamma}:\Gamma_{\gamma}\to C(s(\gamma))$$
is a homeomorphism.
 
\begin{remark}\label{rem: char_functions}
For $\alpha \in V_A^1$ denote by $\chi_{\alpha}$ the characteristic function of $$\{ (x,1,\sigma_A(x)): x\in C(\alpha)\}.$$ For $\alpha = {\o}$ we set $\chi_{{\o}}=1$. Then, for $\alpha=\alpha_1\ldots \alpha_n$ we write $\chi_{\alpha}:= \chi_{\alpha_1}\star \ldots \star \chi_{\alpha_n},$ where $\star$ denotes the convolution of $C_c(\Gamma_A)$. The constructions in \cite[Subsection 3.1]{gerontogiannis25heat} immediately give that the space $\LC$ of locally constant functions on $\Gamma_A$ with compact support is spanned by the following collection; for $\alpha.\beta \in I_A$ and $\nu \in V_A$ such that $\beta \nu \in V_A$, consider $$\chi_{\alpha \beta_{|\beta|} \nu} \star \chi_{\beta \nu}^*=\chi_{s^{-1}_{\alpha.\beta} (C(\beta \nu))}.$$ In particular, for $\nu = {\o}$ the function $\chi_{\alpha \beta_{|\beta|}}\star \chi_{\beta}^*$ is the characteristic function of $\Gamma_{\alpha.\beta}\subset \Gamma_A.$
\end{remark}

\begin{remark}\label{rem: char_functions2}
It is important to note that for $\alpha, \beta\in V_A\setminus \{{\o}\}$ with $\alpha_{|\alpha|} = \beta_{|\beta|}$  the support $$\supp(\chi_{\alpha}\star \chi_{\beta}^*)\subset \Gamma_{\gamma},$$ where $\gamma = \alpha_1\ldots \alpha_{|\alpha|-|\overline{\alpha}\wedge \overline{\beta}|}. \beta_1 \ldots \beta_{|\beta|-|\overline{\alpha}\wedge \overline{\beta}|+1}.$ Here $\overline{\alpha}:=\alpha_{|\alpha|}\ldots \alpha_1$ is the flipped $\alpha$, which in general will not be in $V_A$, and similarly for $\overline{\beta}$. Further, $|\overline{\alpha}\wedge \overline{\beta}|\geq 1$ is the length of the common path of $\overline{\alpha}$ and $\overline{\beta}$. Moreover, we make the convention that $r(\gamma)={\o}$ when $|\alpha|=|\overline{\alpha}\wedge \overline{\beta}|$.
\end{remark}

We now equip $\Gamma_A$ with an Ahlfors regular metric-measure structure. For every $\gamma \in I_A$, we define the pull-back metric $\mathrm{d}_{\gamma}$ on $\Gamma_{\gamma}$ by 
$$\mathrm{d}_{\gamma}(g_1,g_2):=\mathrm{d}(s_{\gamma}(g_1),s_{\gamma}(g_2)).$$ 
In other words, the metric $\mathrm{d}_\gamma$ on $\Gamma_\gamma$ can be defined from declaring $s_\gamma:\Gamma_\gamma\to C(s(\gamma))$ to be an isometry. Then, we define an extended metric $\mathrm{d}_{\Gamma_{A}}$ on $\Gamma_A$, generating the topology on $\Gamma_A$, by 
$$\mathrm{d}_{\Gamma_A}(g_1,g_2):=
\begin{cases}
\mathrm{d}_{\gamma}(g_1,g_2), & \textup{if  for some } \gamma\in I_A \textup{ both}\; g_1,g_2\in \Gamma_\gamma,\\
\infty, & \textup{ otherwise}.
\end{cases}
$$

Next we define the pull-back finite Borel measure $\mu_{\gamma}:=s_\gamma^*\mu$ on $\Gamma_{\gamma}$. For an open set $B\subset \Gamma_{\gamma}$, $\mu_\gamma$ we have 
$$\mu_{\gamma}(B)= \mu(s_{\gamma}(B)).$$ 
The collection $\{\mu_{\gamma}\}_{\gamma\in I_A}$ gives a Borel measure $\mu_{\Gamma_A}$ on $\Gamma_A$. Now since each $s_{\gamma}(\Gamma_{\gamma})=C(s(\gamma))$ is clopen in $\Sigma_A$, and for the latter the measure $\mu$ is Ahlfors $\df$-regular, as mentioned above, we obtain that for every $\gamma \in I_A$, the metric-measure space $(\Gamma_{\gamma},\mathrm{d}_{\gamma},\mu_{\gamma})$ is Ahlfors $\df$-regular with $\df= \log_{\lambda}(\lambda_{\max}).$

\subsection{Log-Laplacian and Hamiltonian operators}\label{sec:G_A_MMS}
The main point of this subsection is to present the analogue of a Dirac operator on $\Gamma_A$. This was developed in \cite{gerontogiannis25heat},  inspired by the logarithmic Laplace--Beltrami operator on Ahlfors regular spaces discussed in \cite{gerontogiannis25dirichlet}. The Dirac operators as above provide Cuntz--Krieger algebras with a differential structure via spectral triples. For the notion of spectral triples we refer to \cite{connes94noncommutative}. 

We now summarize the construction of this so-called \emph{log-Laplacian} and its corresponding Hamiltonian. The construction starts with the $L^{2}$-space: note that 
$$L^2(\Gamma_A, \mu_{\Gamma_A})=\bigoplus_{\gamma\in I_A}L^2(\Gamma_{\gamma},\mu_{\gamma}).$$
For $\gamma\in I_{A}$,  define a positive essentially self-adjoint operator $\Delta_{\gamma}$ acting on locally constant functions $f:\Gamma_{\gamma}\to \mathbb C$ through the formula
$$(\Delta_{\gamma} f)(g) = \int_{\Gamma_{\gamma}} \frac{f(g)-f(h)}{\d_{\gamma}(g,h)^{\df}}\d \mu_{\gamma} (h),\;\;\; g \in \Gamma_\gamma,$$
noting that $\ker \Delta_{\gamma} = \mathbb C \chi_{\Gamma_{\gamma}}$. The \emph{log-Laplacian} is then $\Delta=\bigoplus_{\gamma\in I_A}\Delta_{\gamma}$. In particular, the space $\LC$ is a core for $\Delta$, and for functions $f \in \LC$ we still have the integral representation
\begin{equation}\label{eq:Int_rep}
(\Delta f)(g)=\int_{\Gamma_A} \frac{f(g)-f(h)}{\mathrm{d}_{\Gamma_A}(g,h)^{\df}} \d \mu_{\Gamma_A}(h), \;\; g \in \Gamma_A.
\end{equation}

The operator $\Delta$ admits an explicit spectral decomposition with eigenfunctions forming an orthonormal basis $(\mathrm{e}_{\gamma},\mathrm{e}_{(\gamma,\nu,j)})_{(\gamma,\nu,j)\in \mathfrak{I}_A}\subset \LC$  for $L^2(\Gamma_A,\mu_{\Gamma_A})$, where $\mathfrak{I}_A$ is an appropriate index set -- we refer to \cite[Theorem 3.9] {gerontogiannis25heat} for the details. More specifically,  $$\mathrm{e}_{\gamma}:=\mu(\Gamma_{\gamma})^{-1/2} \chi_{\Gamma_{\gamma}}, \;\;\; \gamma \in I_A,$$ 
and each $\mathrm{e}_{(\gamma,\nu,j)}$ is a \emph{Haar wavelet} supported on $s^{-1}_{\gamma}(C(\nu))$, where in turn $\nu \in V_A \setminus \{{\o}\}$ starts with $s(\gamma)\neq {\o}.$ The kernel  $\ker(\Delta)$ is the closed linear span of $(\mathrm{e}_{\gamma})_{\gamma\in I_A}$ and for $(\gamma,\nu,j)\in \mathfrak{I}_A$ we have 
$$\Delta \mathrm{e}_{(\gamma,\nu,j)}=\lambda_{s(\gamma)}^A(\nu)\mathrm{e}_{(\gamma,\nu,j)},$$ where if $C(\nu)=C(s(\gamma))$ then $\lambda^A_{s(\gamma)}(\nu)=\lambda_{\max} u_{\nu_{|\nu|}}$, and if $C(\nu)\subsetneq C(s(\gamma))$ then
\begin{equation}\label{eq:eigenvalues}
\lambda_{s(\gamma)}^A(\nu)=\lambda_{\max} u_{\nu_{|\nu|}}+\lambda_{\max}\sum_{k=0}^{|\nu|-|s(\gamma)|-1}u_{\nu_{|\nu|-k-1}}(1-P_{\nu_{|\nu|-k-1}, \nu_{|\nu|-k}}).
\end{equation}
Recall that $P$ is the transition probability on the graph underlying $A$.

Although the commutators of $\Delta$ with functions in $\LC$ extend to bounded operators on $L^2(\Gamma_A,\mu_A)$, the operator $\Delta$ does not have compact resolvent, so it does not generate a spectral triple. To this end, we introduce the addition of a potential $V$. To define $V$, first consider the multiplication operator $M_{L}$ determined by a proper continuous `length' function $L:\Gamma_A\to \mathbb{N}_0$ (not to be confused with the transfer map) defined as
\begin{equation}\label{eq:scale}
L|_{\Gamma_\gamma}:=|\gamma|=|r(\gamma)|+|s(\gamma)|, \;\;\; \gamma \in \Gamma_A.
\end{equation}
Note that $M_{L}$ is positive and essentially self-adjoint on ${\rm C_{c}^\infty}(\Gamma_A)$. Consider further the space of finite words with a distinguished ending: 
\begin{equation}\label{eq:Fock}
\tilde{V}_A:=\{\alpha.\beta\in V_A\times \{1,\ldots,N\}: \alpha\beta\in V_A\}=\{\gamma\in I_A: |s(\gamma)|=1\},
\end{equation}
writing $\ell^2(\tilde{V}_A)$ for the associated Hilbert space, and consider the Hilbert space 
$$F_A:=\bigoplus_{\gamma\in \tilde{V}_A}\ker \Delta_{\gamma}\subset L^2(\Gamma_A,\mu_{\Gamma_A}),$$ 
and the associated projection $P_A$ from $L^2(\Gamma_A,\mu_{\Gamma_A})$ onto $F_A$.

We are now ready to complete the construction: the \emph{Hamiltonian} is the operator 
$$D:= - \Delta + V,\qquad V:= (2P_A-1)M_{L}.$$ 
Thus for $f\in \LC$ we have 
\begin{equation}\label{eq:D}
Df=
\begin{cases}
-(\Delta+M_{L})f &\text{if } f\in F_A^{\perp},\\
M_{L}f &\text{if } f\in F_A,
\end{cases}
\end{equation}
and since $D=(2P_A-1)(\Delta + M_{L})$, we have that $|D|=\Delta + M_{L}.$

What is particularly important for us here is that the $1$-eigenspace of $D$ is the finite dimensional subspace $$E_1:= \langle \chi_{\beta}\star \chi_{\beta}^*: \beta \in V_A^1 \rangle,$$ and the $2$-eigenspace is the finite dimensional subspace $$E_2:= \langle \chi_{\alpha \beta} \star \chi_{\beta}^*: \alpha,\beta \in V_A^1\,\, \text{and } \alpha.\beta \in I_A\rangle.$$

Eventually it can be shown (see \cite[Proposition 5.5]{gerontogiannis25heat}) that $(\LC, L^2(\Gamma_A,\mu_{\Gamma_A}), D)$ is a spectral triple on $C^*_r(\Gamma_A)$, where $\LC$ is represented on $L^2(\Gamma_A,\mu_{\Gamma_A})$ via convolution operators and $C^{*}_{r}(\Gamma_{A})$ is the closure of the image of this representation. In general, the operator $D$ introduced above has non-trivial $K$-homology class. As we will not need hereafter any specific result about the structure of $(\LC, L^2(\Gamma_A,\mu_{\Gamma_A}), D)$, we omit the definition of spectral triples.

\subsection{Cuntz--Krieger algebras}

The groupoid picture outlined above is practical for the definition of the operator $D$. In order though to study its interaction with compact quantum group theory, we have to explain the so-called \emph{Cuntz--Krieger} picture introduced in \cite{cuntz80class}. This will also bridge our construction of the quantum isometry groups with the theory of quantum automorphism groups of directed graphs.

The Cuntz--Krieger algebra $O_{A}$ is the universal $C^{*}$-algebra generated by the elements $S_{1}, \ldots, S_N$  subject to the relations
\begin{equation}\label{eq: CK-relations} 
\sum_{j=1}^{N}S_{j}S_{j}^{*}=1 \quad \mbox{and}\quad  S_{i}^{*}S_{k}=\delta_{ik}\sum_{j=1}^{N}A_{ij}S_{j}S_{j}^{*},  \;\;i,k=1,\ldots, N.
\end{equation}
The universal $C^{*}$-algebra $O_{A}$ carries an action of the group $\mathbb{T}$ by $*$-automorphisms, known as the \emph{gauge action} and determined by the formula 
\begin{equation}
\label{eq:gauge-action}
\alpha_{z}(S_{i}):=zS_{i},\quad z\in\mathbb{T},\,\,i\in\{1,\cdots, N\}.
\end{equation} 

There is a canonical \emph{KMS state} on $O_{A}$ with respect to the gauge action, which we will denote by $\tau$. Further, $\rho:O_A\to B(L^2(O_A,\tau))$ shall denote the corresponding faithful GNS representation with $\xi\in L^2(O_A,\tau)$ being the associated cyclic vector, so that $\tau(x)=\left\langle \xi, \rho(x)\xi \right\rangle$, $x \in O_A$. Using the Cuntz--Krieger relations and the $\log(\lambda_{\max})$-KMS condition of $\tau$ we obtain that for every $\alpha,\beta\in V_A\setminus \{{\o}\}$
\begin{equation}\tau(S_{\alpha}S^*_{\beta})=\frac{\delta_{\alpha,\beta}}{\lambda_{\max}^{|\alpha|}}\sum_{j=1}^{N} A_{\alpha_{|\alpha|},j}\tau(S_jS_j^*),
\label{tauKMS}	\end{equation} 
where naturally $S_\alpha:=S_{\alpha_1}\cdots S_{\alpha_n}$, and similarly for $S_\beta$.

The $C^*$-algebra $O_A$ is canonically isomorphic to $C^*_r(\Gamma_A)$, studied in the last subsection, through the map $$S_i\mapsto \chi_i \star (\cdot).$$ This isomorphism is also equivariant for the respective $\mathbb T$-actions. Under this isomorphism the $^*$-algebra generated by $\{S_i: 1\leq 1 \leq N\}$ is identified with the space of locally constant functions with compact support $C_c^{\infty}(\Gamma_A)$. Moreover, for $f\in C_c(\Gamma_A)$ we have 
\begin{equation}\label{eq:KMS}
\tau(f)=\int_{\Sigma_A} f(x,0,x) \d \mu (x).
\end{equation}
Finally, $L^2(O_A,\tau)$ can be identified with $L^2(\Gamma_A,\mu_{\Gamma_A})$ by interpreting (for $i=1, \ldots, N$) the vector $\rho(S_i)\xi $ as the function $\chi_i$. That being said, we will sometimes view $D$ as an operator that acts on $L^2(O_A,\tau)$ instead of $L^2(\Gamma_A,\mu_{\Gamma_A})$.

We shall also need one more well-known fact: the closed linear span of $\{S_\alpha S_\alpha^*: \alpha \in V_A\}$ inside $O_A$ is a commutative $C^*$-algebra which is isomorphic to $C(\Sigma_A)$. The isomorphism maps each  $S_\alpha S_\alpha^*$ to the characteristic function of the appropriate cylinder; that is, $\chi_{C(\alpha)}$.

\section{Quantum isometries of $O_A$}\label{sec:Quantum_isometries}

The goal of this work is to understand the quantum symmetries of the spectral triple introduced above. The classical isometry group was computed in \cite{gerontogiannis25heat} and shown to be isomorphic to the permutational wreath product $\mathbb{T}\wr\mathrm{Aut}(A)$, where $\mathrm{Aut}(A)$ is the automorphism group of the directed graph with adjacency matrix $A$, naturally viewed as a subgroup of the permutation group $\mathbb S_N$ (see also the last section of this paper). In this section, we will recall some basic definitions and facts on compact quantum groups and establish some first results on their isometric actions on $O_{A}$. We will also introduce a new family of compact quantum groups, which will then be used to describe the quantum isometries of the spectral triple on $O_{A} = C^{*}_{r}(\Gamma_{A})$.

\subsection{Compact quantum groups}

We refer the reader to \cite{neshveyev13compact} for a comprehensive treatment of the theory of compact quantum groups and proofs of the results mentioned below. Also, we follow the convention that tensor products of $C^*$-algebras are always minimal. 

A \emph{compact quantum group} $\mathbb{G} = (\CG, \Phi)$ is given by a $C^*$-algebra $\CG$ and a $*$-homomorphism $\Phi : \CG\to \CG\otimes \CG$ such that
\begin{itemize}
\item $(\Phi\otimes \mathrm{id})\circ\Phi = (\mathrm{id}\otimes \Phi)\circ\Phi$;
\item The linear spans of $(1\otimes \CG)\Phi(\CG)$ and $(\CG\otimes 1)\Phi(\CG)$ are dense in $\CG\otimes \CG$.
\end{itemize}

It is a fundamental fact of the theory that there is a unital dense $*$-subalgebra $\Pol(\QG)$ of $\CG$ which admits a Hopf $*$-algebra structure via the comultiplication $\Phi$ and the multiplication $m$ of $\CG$, so that in particular $\Phi(\Pol(\QG)) \subset \Pol(\QG) \odot \Pol(\QG)$. We will denote its counit by $\varepsilon$ and its antipode by $k$. Note that $k$ is invertible with $k^{-1} = *\circ k \circ *$, and that $\Pol(\mathbb G)$ is also a Hopf $*$-algebra for the structure maps $\Phi, m^{\op},\varepsilon,k^{-1}.$

\begin{remark}
In the sequel, given any compact quantum group $\mathbb G$ we shall view $\CG \otimes L^2(O_A,\tau)$ as a right Hilbert $\CG$-module, with the module structure given by the right multiplication $\cdot$ in the $\CG$-coordinate. The symbol $B_{\CG}(\CG\otimes L^2(O_A,\tau))$ will denote the space of all adjointable linear operators on this Hilbert module. Further, $\Pol(\mathbb G)$ denotes the span of the coefficients of finite dimensional unitary representations of $\mathbb G$.
\end{remark}

Let us now give a formal definition of an action of a compact quantum group $\QG$ on a Cuntz--Krieger algebra $O_{A}$, since this will be our main object of study.

\begin{definition}\label{def:cqgaction}
A \emph{(left) action} of $\QG$ on $O_{A}$ is a unital $*$-homomorphism $$\varphi:O_A\to C(\mathbb G)\otimes O_A$$ such that
\begin{itemize}
\item $(\Phi \otimes \mathrm{id})\circ\varphi = (\mathrm{id}\otimes \varphi)\circ\varphi$;
\item $\varphi(O_{A})(\CG\otimes 1)$ is dense in $\CG\otimes O_{A}$.
\end{itemize}
\end{definition}

It will be important for us to be able to implement such an action through conjugation by a unitary operator at the level of $L^{2}$-spaces. We show below that this is possible provided the canonical KMS state is preserved. This is well-known, but we write down the details to fix the conventions we are using. Recall that $(\rho, L^2(O_A, \tau), \xi)$ is the GNS triple of the KMS state $\tau$.

\begin{lemma}\label{lem:preserveKMS_1}
Let $\varphi:O_A\to C(\mathbb G)\otimes O_A$ be an action of a compact quantum group $\mathbb G$ that preserves the KMS state $\tau$; i.e. for $T\in O_A$, $$(\id \otimes \tau)(\varphi(T))=\tau(T)1.$$
Then, the linear map $V_{\varphi}:L^2(O_A,\tau)\to \CG \otimes L^2(O_A,\tau)$ defined for $\rho(T)\xi\in \rho(O_A)\xi$ by 
\begin{equation}\label{eq:preserveKMS_1_0}
V_{\varphi}(\rho(T)\xi):=(\id\otimes \rho)(\varphi(T))(1\otimes \xi)
\end{equation}
yields a unitary representation of $\mathbb G$ on $L^2(O_A,\tau)$ through the operator $U_{\varphi}\in B_{\CG}(\CG\otimes L^2(O_A,\tau))$ given on $g\otimes \rho(T)\xi \in  \CG \odot \rho(O_A)\xi$ by 
\begin{equation}\label{eq:preserveKMS_1_1}
U_{\varphi}(g \otimes \rho(T)\xi):=V_{\varphi}(\rho(T)\xi)\cdot g.
\end{equation}
Moreover, $U_{\varphi}$ implements $\varphi$, in the sense that, for $T\in O_A$ we have 
$$(\id\otimes \rho)(\varphi(T))=U_{\varphi}(1\otimes \rho(T))U_{\varphi}^*.$$ 
\end{lemma}

\begin{proof}
For every $T\in O_A$, the operator $(1\otimes \rho)(\varphi(T))\in \CG \otimes B(L^2(O_A,\tau))$ is adjointable, with adjoint $(1\otimes \rho)(\varphi(T^*))$, as $\CG \otimes B(L^2(O_A,\tau))$ embeds in $B_{\CG}(\CG\otimes L^2(O_A,\tau)).$ Then, from the KMS preservation we see that for every $T,R \in O_A$, $$\langle V_{\varphi}(\rho(T)\xi), V_{\varphi}(\rho(R)\xi) \rangle_{\CG}=\langle \rho(T)\xi, \rho(R)\xi \rangle_{L^2(O_A,\tau)}1.$$ In other words, $V_{\varphi}$ is an isometry and hence $U_{\varphi}$ is an isometry on $\CG\otimes L^2(O_A,\tau)$. Also, since the linear span of $\varphi(O_A) (\CG \otimes 1)$ is dense in $\CG \otimes O_A$, it immediately follows that $U_{\varphi}$ has dense range, hence it is a unitary in $B_{\CG}(\CG\otimes L^2(O_A,\tau))$. The fact that $U_\varphi$ satisfies $(\Phi \otimes \id) (U_{\varphi})= (U_{\varphi})_{13}(U_{\varphi})_{23}$ is clear since $\varphi$ is a left action. Finally, it is straightforward to show that $(\id\otimes \rho)(\varphi(T))U_{\varphi}=U_{\varphi}(1\otimes \rho(T))$, for all $T\in O_A$.
\end{proof}

We can now define the crucial notion of this work, that of an isometric compact quantum group action with respect to the spectral triple on $O_A$. Several definitions of an isometric action of a quantum group have appeared in the literature, and we will use the one which fits best our framework. Note that as the eigenspaces of the Dirac operator we consider may in fact be viewed as finite subspaces spanning a dense subalgebra of the Cuntz-Krieger algebra $O_A$, we could also equivalently phrase the isometry requirement in terms of the orthogonal filtrations of \cite{banicaskalski}. We refer to \cite{bhowmick16quantum} or to \cite{Adam} for a detailed exposition of the subject.

\begin{definition}\label{def:D_isom}
An action $\varphi:O_A\to C(\mathbb G)\otimes O_A$ of a compact quantum group $\mathbb G$ will be called $D$-isometric if
\begin{enumerate}
\item $\varphi$ preserves the KMS state $\tau$ on $O_A$, i.e.\ $(\id \otimes \tau) \varphi = \tau(\cdot)1$;
\item $U_{\varphi}(\Dom (1\otimes D))=\Dom (1\otimes D)$, where $\Dom (1\otimes D)$ denotes the domain of $1 \otimes D$;
\item $U_{\varphi}(1\otimes D)U_{\varphi}^*=1\otimes D$ on $\Dom (1\otimes D)$.
\end{enumerate}
\end{definition}

Let us say that a unital $*$-homomorphism $\pi : C(\QG)\to C(\QH)$ is a \emph{quantum group homomorphism} (from $\QH$ to $\QG$) if $\Phi_{\QH}\circ\pi = (\pi\otimes \pi)\circ\Phi_{\QG}$. Consider now the category $Q_A(D)$ with objects $(\mathbb G,\varphi)$ being $D$-isometric actions $\varphi:O_A\to C(\mathbb G)\otimes O_A$ and where the morphisms $(\mathbb G_2,\varphi_2)\to (\mathbb G_1,\varphi_1)$ are given by compact quantum group homomorphisms $\pi:C(\mathbb G_1) \to C(\mathbb G_2)$ such that the following diagram commutes:
\begin{equation}\label{eq:1}
\begin{tikzcd}
  & \qquad C(\mathbb G_1)\otimes O_A \arrow[dr,"\pi\otimes \id"] \\
O_A \arrow[ur,"\varphi_1"] \arrow[rr,"\varphi_2"] &&  C(\mathbb G_2)\otimes O_A
\end{tikzcd}.
\end{equation}
Our goal is to prove the existence of a final object in $Q_A(D)$ which we shall call  \emph{the (compact) quantum isometry group of the spectral triple $(\LC, L^2(O_A,\tau),D)$}. Once we establish its existence, it will be clear that it is also unique, up to a unique isomorphism in $Q_A(D)$. Note that we do include \emph{compactness} of the quantum isometry group we are considering as a part of the definition, whereas in the case of classical actions considered in \cite{gerontogiannis25heat} it was not assumed a priori, but only followed as a consequence of the computation of the isometry group. As the example related to the Toeplitz algebra studied in \cite[Section 4]{Park95Isometries} shows, in general even a classical isometry group of a noncommutative manifold need not be compact. On the other hand, the very existence of the universal object in the category we consider, suggests that the resulting compact quantum group is a relevant notion to study; one can show that in the example treated in \cite{Park95Isometries} mentioned above, a corresponding universal \emph{compact} isometry group does not exist. Finally, we could also consider more generally the category of all  isometric actions on $O_A$ of \emph{locally compact} quantum groups. Adapting the proofs of the next section to that context would show that the resulting universal object would still be the same; in particular it would be compact. Hence we work in the compact setting from the very beginning.

Let us now record elementary properties of $D$-isometric actions.

\begin{lemma}\label{lem:preserveKMS_2}
For every $(\mathbb G,\varphi)\in Q_A(D)$, 
$$\varphi(C_c^{\infty}(\Gamma_A))\subset \Pol(\mathbb G)\odot C_c^{\infty}(\Gamma_A).$$
 Therefore, the map $(k^{-1}\otimes \id)V_{\varphi}:\rho(C_c^{\infty}(\Gamma_A))\xi\to \Pol(\mathbb G) \odot \rho(C_c^{\infty}(\Gamma_A))\xi$ is well-defined and extends to a right $\Pol(\mathbb G)$-module map $U_{\varphi,k}$ acting on $\Pol(\mathbb G)\odot \rho(C_c^{\infty}(\Gamma_A))\xi$ by $$U_{\varphi,k}(g\otimes \rho(T)\xi):= (k^{-1}\otimes \id) V_{\varphi}(\rho(T)\xi) \cdot g.$$ Moreover, the map $U_{\varphi,k}$ is a restriction of the unitary $U_{\varphi}^*\in B_{\CG}(\CG\otimes L^2(O_A,\tau))$.
\end{lemma}

\begin{proof}
Let $\spec(D)$ denote the spectrum of $D$, which consists of eigenvalues. For $\zeta \in \spec(D)$ denote by $E_{\zeta}$ the finite dimensional $\zeta$-eigenspace. As explained in Subsection \ref{sec:G_A_MMS}, each space $E_{\zeta}$ has an orthonormal basis $f_{\zeta,1},\ldots,f_{\zeta,\dim E_{\zeta}}$ of Haar wavelets. 

Since $U_{\varphi}$ commutes with $1\otimes D$ we have that $U_{\varphi}$ preserves $\CG \otimes E_{\zeta}$. Hence, the restriction of $U_{\varphi}$ to $\CG \otimes E_{\zeta}$ is a finite dimensional unitary representation of $\mathbb G$ and so, for every Haar wavelet $\rho(f)\xi \in E_{\zeta}$, the element $U_{\varphi}(1\otimes \rho(f)\xi)$ has coefficients in $\Pol(\mathbb G).$ In turn, as $\rho(C_c^{\infty}(\Gamma_A))\xi$ is spanned by the Haar wavelets of all $E_{\zeta}$, we get that for every $f\in C_c^{\infty}(\Gamma_A)$ there are elements $g_1, \ldots, g_n \in \Pol(\mathbb G)$ and $f_1,\ldots, f_n \in C_c^{\infty}(\Gamma_A)$ such that 
\begin{equation}\label{eq:preserveKMS_2_0}
U_{\varphi}(1\otimes \rho(f)\xi)= \sum_{j=1}^ng_{j}\otimes \rho(f_j)\xi.
\end{equation}
Moreover, $U_{\varphi}^*$ commutes with $1\otimes D$, as the latter is self-adjoint. Hence, the analogue of the formula \eqref{eq:preserveKMS_2_0} holds with $U_{\varphi}^*$ in place of $U_{\varphi}$ as well. This means that
 $$U_{\varphi}(\Pol(\mathbb G) \odot \rho(C_c^{\infty}(\Gamma_A))\xi)=\Pol(\mathbb G) \odot \rho(C_c^{\infty}(\Gamma_A))\xi.$$

Since the vector $1\otimes \xi$ is cyclic for $1\otimes \rho: \CG \otimes O_A \to B_{\CG}(C(\mathbb G)\otimes L^2(O_A,\tau))$ and $\tau$ is a faithful state (hence $1\otimes \tau:\CG \otimes O_A\to \CG$ is faithful as well), we obtain that  for $f,g_1, \ldots, g_n$, $f_1, \ldots, f_n$ as above 
\begin{equation}\label{eq:preserveKMS_2_1}
\varphi(f)=\sum_{j=1}^ng_j\otimes f_j,
\end{equation}
and recalling \eqref{eq:preserveKMS_1_0} we see that $U_{\varphi,k}$ is well-defined.  

Now let $f\in C_c^{\infty}(\Gamma_A)$ and we will show that $U_{\varphi,k}U_{\varphi}(1\otimes \rho(f)\xi)=1\otimes \rho(f)\xi$, so that $U_{\varphi}^*$ agrees with $U_{\varphi,k}$ on $\Pol(\mathbb G) \odot \rho(C_c^{\infty}(\Gamma_A))\xi$. Using the notation in \eqref{eq:preserveKMS_2_1} we have

\begin{align*}
U_{\varphi,k}U_{\varphi}(1\otimes \rho(f)\xi)&= \sum_{j=1}^n U_{\varphi,k}(1\otimes \rho(f_j)\xi)\cdot g_j = \sum_{j=1}^n \left[(\id\otimes \rho) (k^{-1}\otimes \id)(\varphi(f_j))(1\otimes \xi) \right] \cdot g_j\\
&=(\id\otimes \rho) \left(\sum_{j=1}^n (k^{-1}\otimes \id)(\varphi(f_j)) (g_j\otimes 1)\right) (1\otimes \xi)\\
&=(\id\otimes \rho)\left((m^{\op}\otimes \id)(\id\otimes k^{-1}\otimes \id)(\id\otimes \varphi)(\varphi(f))\right)(1\otimes \xi)\\
&= (\id\otimes \rho)\left((m^{\op}\otimes \id)(\id\otimes k^{-1}\otimes \id)(\Phi\otimes \id)(\varphi(f))\right)(1\otimes \xi)\\
&=(\id\otimes \rho) \left((\varepsilon 1 \otimes \id)(\varphi (f))\right)(1\otimes \xi) =(\id\otimes \rho)(1\otimes f)(1\otimes \xi).
\end{align*}
\end{proof}

\subsection{Ariadne's thread}

We now want to define the compact quantum groups which will appear as quantum isometry groups of $O_{A}$. To do this, it will prove practical to have a whole family of compact quantum groups indexed by positive integers. Recall that in all this work, $A\in M_{N}(\{0, 1\})$ is a primitive matrix. We first define the $C^*$-algebras associated to our compact quantum groups. We will use the following standard terminology introduced in \cite{banica05quantum}: a matrix $p = (p_{ij})_{1\leqslant i, j\leqslant N}$ with coefficients in a $C^*$-algebra is a \emph{magic unitary} if the coefficients are self-adjoint projections such that their sum in any row or column equals the unit (equivalently $p$ is a unitary matrix whose entries are self-adjoint projections).

\begin{definition}\label{defn: Ariadne_0}
Let $\ell\in \mathbb N \cup \{\infty\}$ and consider the universal $C^*$-algebra $\CA$ generated by partial isometries $u=(u_{\alpha, \beta})_{ \alpha,\beta \in V_A^1}$, with range projections $p=(p_{\alpha, \beta})_{\alpha,\beta \in V_A^1}$ and source projections $q=(q_{\alpha, \beta})_{\alpha,\beta \in V_A^1}$, such that

\begin{enumerate}[(I)]
\item $p,q$ are magic unitaries preserving the right Perron--Frobenius eigenvector; i.e. $p\vec{u}=\vec{u}1=q\vec{u}$;
\item $A$ intertwines $p$ and $q$; i.e. $Ap=qA$;
\item for $n \leq \ell$ and $\alpha, \beta\in V_A^{n}$, the element $u_{\alpha,\beta}:=u_{\alpha_1,\beta_1}\ldots u_{\alpha_n,\beta_n}$ is a partial isometry. The associated range and source projections are denoted by $p_{\alpha,\beta}$ and $q_{\alpha,\beta}$.
\end{enumerate}
\end{definition}
It is clear that for $\ell \leq \ell'$, $C(\QG_{A}^{\ell'})$ is a quotient of $\CA$. Therefore, any algebraic relation that holds in $\CA$ also holds in $C(\QG_{A}^{\ell'})$. Also if $p\vec{u}=\vec{u}1$, property II implies that $q\vec{u}=\vec{u}1$ as well. Later in Section 5 we prove that, in the special case where $A$ is filled with $1$'s, these quantum groups coincide with a class introduced in \cite{Mang2022phd}.

\begin{remark}\label{rem: PFvector}
The second condition in (I) follows from the first (the magic unitarity property of $p$ and $q$) if the right Perron--Frobenius eigenvector is constant, i.e.\ if the graph given by the matrix $A$ is \emph{out-regular}. We do not know whether in general it follows from the other conditions -- see the discussion in Subsection \ref{Sec:classical}.
\end{remark}

\begin{remark}\label{rem: orthogonality}
For the generators of $C(\QG_A^{1})$ we have $$u_{\beta,\alpha}(u_{\delta,\alpha})^*=(u_{\beta,\alpha})^*u_{\delta,\alpha}= u_{\alpha, \beta}(u_{\alpha,\delta})^*= (u_{\alpha, \beta})^*u_{\alpha,\delta}=0,$$ whenever $\alpha, \beta,\delta \in V_A^1$ with $\beta\neq \delta$. This is because every row and column of $p,q$ consists of projections summing to $1$. Such orthogonality conditions associated to longer admissible words $\alpha,\beta,\delta$ hold in $\CA$ for large enough $\ell$, as we will see in Lemma \ref{lem: words_of_s}.
\end{remark}

\begin{remark}\label{rem:distinct}
A priori it is not clear whether $\QG_{A}^{\ell}\neq \QG_{A}^{\ell'}$ for $\ell\neq \ell'$. In Corollary \ref{cor:diff_Gl} we shall see that this is the case for $A=\textbf{1}_N$.
\end{remark}

Our first task is to construct a coproduct $\Phi_{A}$ on the $C^*$-algebra $C(\QG_{A}^{\ell})$, so to endow it with a compact quantum group structure. To do this, we first establish some algebraic relations between the generators. First, one can consider  products of generators for any pair of words on $\{1, \ldots, N\}$, and these products will vanish under certain admissibility conditions.

\begin{lemma}\label{lem: words_of_s_0}
For every $n\in \mathbb N$, $\ell\in \mathbb N \cup \{\infty\}$ and $\alpha',\beta' \in \{1,\ldots, N\}^n$, if $\alpha' \in V_A$ and $\beta' \not\in V_A$, or if $\alpha' \not\in V_A$ and $\beta' \in V_A$, in $\CA$ it holds that $u_{\alpha',\beta'}:=u_{\alpha'_1,\beta'_1}\ldots u_{\alpha'_n,\beta'_n}=0.$
\end{lemma}

\begin{proof}
It suffices to prove it for $\ell=1$, and we claim that it is also enough to prove it for $n=2$. Indeed, the assumption means that there exists $i\in \N$ such that $\alpha'_{i}\alpha'_{i+1}$ is admissible but $\beta'_{i}\beta_{i+1}'$ is not, or the converse. The case $n = 2$ then yields $u_{\alpha'_{i}\beta'_{i}}u_{\alpha'_{i+1}\beta'_{i+1}} = 0$, from which the vanishing of the whole product follows.

From property II in $C(\QG_A^{1})$ we have $$\sum_{\delta} A_{\alpha_1, \delta}p_{\delta, \beta_2}=\sum_{\delta} q_{\alpha_1, \delta} A_{\delta, \beta_2}.$$ Multiplying by $q_{\alpha_1,\beta_1}$ from the left and by $p_{\alpha_2,\beta_2}$ from the right, by Remark \ref{rem: orthogonality} we get $$A_{\alpha_1,\alpha_2}q_{\alpha_1,\beta_1}p_{\alpha_2,\beta_2}=q_{\alpha_1,\beta_1}p_{\alpha_2,\beta_2}A_{\beta_1,\beta_2}.$$
Therefore, if $A_{\alpha_{1},\alpha_{2}} = 1$ and $A_{\beta_{1},\beta_{2}} = 0$, then $q_{\alpha_{1},\beta_{1}}p_{\alpha_{2},\beta_{2}} = 0$ and the result follows from the equality $u_{\alpha_1,\beta_1}u_{\alpha_2,\beta_2}=u_{\alpha_1,\beta_1}q_{\alpha_1,\beta_1}p_{\alpha_2,\beta_2}u_{\alpha_2,\beta_2}$. The other case can be established similarly.
\end{proof}

As a consequence, we obtain more magic unitary matrices given by products of generators. Recall the $\wedge$ notation introduced first in Remark \ref{rem: char_functions2}.

\begin{lemma}\label{lem: words_of_s}
Let $\ell\in \mathbb N \cup \{\infty\}$ and $n\leq \ell$. Then, for $\CA$ we have that
\begin{enumerate}
\item the matrices $(p_{\alpha, \beta})_{\alpha,\beta \in V_A^n}$ and $(q_{\alpha, \beta})_{\alpha,\beta \in V_A^n}$ are magic unitaries;
\item for $\alpha,\alpha',\beta,\beta' \in V_A^n$, if $|\alpha \wedge \alpha'| \neq |\beta \wedge \beta'|$ then $p_{\alpha,\beta}p_{\alpha',\beta'}=q_{\alpha,\beta}q_{\alpha',\beta'}=0$. 
\end{enumerate}
\end{lemma}

\begin{proof}
For part (1), pick $\alpha \in V_A^n$. Then, from Lemma \ref{lem: words_of_s_0} and the fact that the matrix $p$  is a magic unitary we have 
\begin{align*}
\sum_{\beta\in V_A^n} p_{\alpha,\beta}&=\sum_{\beta \in \{1,\ldots, N\}^n} p_{\alpha,\beta}=1,\\
\sum_{\beta\in V_A^n} p_{\beta,\alpha}&=\sum_{\beta \in \{1,\ldots, N\}^n} p_{\beta,\alpha}=1.
\end{align*}
Similarly, as $q$ is a magic unitary, $$\sum_{\beta\in V_A^n} q_{\alpha,\beta}=\sum_{\beta\in V_A^n} q_{\beta,\alpha}=1,$$
proving the first point. As for the second point, observe that the orthogonality relations for coefficients of a magic unitary imply
\begin{equation}\label{eq:orthogonality2}
u_{\beta,\alpha}(u_{\delta,\alpha})^*=(u_{\beta,\alpha})^*u_{\delta,\alpha}= u_{\alpha, \beta}(u_{\alpha,\delta})^*= (u_{\alpha, \beta})^*u_{\alpha,\delta}=0,
\end{equation}
from which the result follows.
\end{proof}

With this in hand, we can provide a natural compact quantum group structure on each $\CA$.

\begin{prop}\label{prop:Qgroup}
Let $\ell\in \mathbb N\cup \{\infty\}$. Then, the unital map $\Phi_A: \CA \to \CA \otimes \CA$ given on generators by $$\Phi_A(u_{\alpha, \beta}) := \sum_{\gamma=1}^{N} u_{\alpha, \gamma} \otimes u_{\gamma, \beta}$$ extends to a $*$-homomorphism, turning $(\CA, \Phi_A)$ into a compact quantum group of Kac type. 
\end{prop}

\begin{proof}
First observe that by the defining relations, the matrix $u=(u_{\alpha, \beta})_{ \alpha,\beta \in V_A^1}$ over $\CA$ and its conjugate $\overline{u}:=((u_{\alpha, \beta})^*)_{\alpha,\beta\in V_A^1}$ are unitaries. Moreover, set
\begin{equation*}
U_{\alpha, \beta} := \sum_{\gamma\in V_{A}^1}u_{\alpha,\gamma}\otimes u_{\gamma,\beta}.
\end{equation*}
Then, a direct computation shows that it also satisfies property (II) of Definition \ref{defn: Ariadne_0}. Also, Lemma \ref{lem: words_of_s} for $n=2$, it shows that $U$ satisfies property (I). Property (III) follows from the orthogonality relations given by Lemma \ref{lem: words_of_s}. The universal property of $C(\QG_{A}^{\ell})$ therefore ensures the existence of $\Phi_{A}$. As the coefficients of $u$ generate the $C^*$-algebra, we can conclude by \cite[Prop 1.1.4]{neshveyev13compact}.
\end{proof}

\begin{definition}
Given $\ell \in \mathbb N\cup \{\infty\}$, we will refer to $\GA$ as the $\ell$-\textit{Ariadne quantum group}. The multiplication, counit and antipode on the Hopf $*$-algebra $\Pol (\GA)$ will be denoted by $m_A,\varepsilon_A, k_A$ as soon as they are clear from the context. Also, the last of these maps is given by $k_A(u_{\alpha,\beta})=(u_{\beta,\alpha})^*.$ 
\end{definition}

\section{The main result}\label{sec:main}

In this section we will prove the main result of our work, which is Theorem \ref{thm:factoring}, saying that $\QG_{A}^{\infty}$ is the universal compact quantum group acting $D$-isometrically on $O_{A}$. The proof naturally splits into two parts: first, we show that $\QG_{A}^{\infty}$ acts $D$-isometrically on $O_{A}$, and second we prove that any $D$-isometric action factors through this action.

\subsection{Isometry of the action}

We start by defining a natural action. To prove that it is $D$-isometric, we will need to find a unitary implementation, which will be automatic if the KMS state is preserved. We will therefore also prove that fact at the same time. We begin with a direct consequence of property I.

\begin{prop}\label{prop:equal_volume}
In $\CAi$, for every $n\in \mathbb N$ and $\alpha,\beta \in \{1,\ldots, N\}^n$, if $u_{\alpha,\beta}\neq 0$ then $\tau(S_{\alpha}S_{\alpha}^*)=\tau(S_{\beta}S_{\beta}^*).$ 
\end{prop}

\begin{proof}
Consider a non-zero coefficient $u_{\alpha,\beta}$. In particular, $u_{\alpha_{|\alpha|},\beta_{|\beta|}}\neq 0$, or equivalently $p_{\alpha_{|\alpha|},\beta_{|\beta|}}\neq 0$. By property I, the associated entries of the Perron--Frobenius eigenvector $\vec{u}$ are equal, namely $u_{\alpha_{|\alpha|}}=u_{\beta_{|\beta|}}$. Since $\alpha,\beta$ have the same length we then get $\tau(S_{\alpha}S_{\alpha}^*)=\tau(S_{\beta}S_{\beta}^*).$
\end{proof}

\begin{prop}\label{prop:actions}
For every $\ell\in \mathbb N \cup \{\infty\}$, the compact quantum group $\GA$ acts on $O_A$ via the $*$-homomorphism $\varphi_A:O_A\to \CA \otimes O_A$, defined on generators as 
$$\varphi_A(S_{\alpha}):=\sum_{\alpha'=1}^N u_{\alpha,\alpha'} \otimes S_{\alpha'}, \;\;\; \alpha =1, \ldots, N.$$ 
Moreover, for $\ell=\infty$ the action preserves the KMS state $\tau$.
\end{prop}

\begin{proof}
It suffices to prove the first part for $\ell=1$. The map $\varphi_A$ extends to a $*$-homomorphism on $O_A$ if the family $\{\varphi_A(S_{\alpha})\}_{\alpha \in V_A^1}$ satisfies the Cuntz--Krieger relations. Indeed, due to orthogonality (see Remark \ref{rem: orthogonality}) for all $\alpha \in V_A^1$ we have $$\varphi_A(S_\alpha)\varphi_A(S_{\alpha})^*=\sum_{\alpha'=1}^Np_{\alpha, \alpha'}\otimes S_{\alpha'}S_{\alpha'}^*.$$ Then, it is clear that 
$$\sum_{\alpha=1}^N\varphi_A(S_\alpha)\varphi_A(S_{\alpha})^*=1.$$ 
Let now $\beta\in \{1, \ldots, N\}$, $\beta\neq \alpha$. Then similarly,
\begin{align*}
\varphi_A(S_{\alpha})^*\varphi_A(S_{\beta})&= \left(\sum_{\alpha'=1}^N (u_{\alpha,\alpha'})^*\otimes S_{\alpha'}^*\right) \left(\sum_{\beta'=1}^N u_{\beta,\beta'}\otimes S_{\beta'}\right)=\sum_{\alpha'=1}^N(u_{\alpha,\alpha'})^*u_{\beta,\alpha'}\otimes S^*_{\alpha'}S_{\alpha'}=0.
\end{align*}
Moreover, 
\begin{align*}
\varphi_A(S_{\alpha})^*\varphi_A(S_{\alpha})&= \sum_{\beta=1}^Nq_{\alpha,\beta}\otimes S^*_{\beta}S_{\beta}
=\sum_{\beta=1}^N q_{\alpha,\beta}\otimes \left(\sum_{\beta'=1}^N A_{\beta,\beta'} S_{\beta'}S_{\beta'}^*\right)\\
&=\sum_{\beta'=1}^N\left(\sum_{\beta=1}^N q_{\alpha,\beta}A_{\beta,\beta'}\right)\otimes S_{\beta'}S_{\beta'}^*
=\sum_{\beta'=1}^N\left(\sum_{\beta=1}^NA_{\alpha,\beta}p_{\beta,\beta'}\right)\otimes S_{\beta'}S_{\beta'}^* \\&=\sum_{\beta=1}^NA_{\alpha,\beta} \varphi_A(S_{\beta})\varphi_A(S_{\beta})^*.
\end{align*}

Furthermore, one can see that $(\varepsilon_A \otimes \mathrm{id}) \circ \varphi_A = \mathrm{id}$ and $(\Phi_A \otimes \mathrm{id})\circ \varphi_A=(\mathrm{id} \otimes \varphi_A)\circ \varphi_A$ by checking both equalities on the generators of $O_A$. This implies the second condition in Definition \ref{def:cqgaction} (see for instance the proof of \cite[Lem 2.13]{commer17actions}), hence the result.

As for the second part of the statement, assume $\ell=\infty$. Then, for the KMS preservation it suffices to show that for $\alpha.\beta \in I_A$ and $\nu\in V_A$ such that $\beta \nu \in V_A$, one has
\begin{equation}\label{eq:KMS_pres}
(\id\otimes \tau)\varphi_A(S_{\alpha\beta_{|\beta|}\nu}S_{\beta \nu}^*)=\tau(S_{\alpha\beta_{|\beta|}\nu}S_{\beta \nu}^*)1,
\end{equation}
as those $S_{\alpha\beta_{|\beta|}\nu}S_{\beta \nu}^*$ span $\LC$. Indeed, we have
$$\varphi_A(S_{\alpha\beta_{|\beta|}\nu}S_{\beta \nu}^*)=\sum_{\alpha',\beta',\nu',\delta} u_{\alpha\beta_{|\beta|}\nu,\alpha'\delta\nu'}(u_{\beta \nu, \beta' \nu'})^*\otimes S_{\alpha'\delta \nu'}S_{\beta' \nu'}^*,$$ 
where $\alpha',\beta',\nu',\delta\in V_A, |\alpha'|=|\alpha|, |\beta'|=|\beta|, |\nu'|=|\nu|, |\delta|=|\beta_{|\beta|}|=1$ and $\alpha'\delta\nu',\beta'\nu'\in V_A$. In fact, due to Lemma \ref{lem: words_of_s} (in particular \eqref{eq:orthogonality2}) we can simplify the above sum to
\begin{equation}\label{eq:phi_A}
\varphi_A(S_{\alpha\beta_{|\beta|}\nu}S_{\beta \nu}^*)=\sum_{\alpha',\alpha'',\nu',\delta} u_{\alpha\beta_{|\beta|}\nu,\alpha'\delta\nu'}(u_{\beta \nu, \alpha'' \delta \nu'})^*\otimes S_{\alpha'\delta \nu'}S_{\alpha''\delta \nu'}^*,
\end{equation}
where $\alpha',\nu',\delta$ are as before, $\alpha''\in V_A$ with $|\alpha''|=|\beta|-1, \alpha''\delta\in V_A$ and $\alpha'_{|\alpha'|}\neq \alpha''_{|\alpha''|}.$ 

Therefore, if $\alpha\neq {\o}$ then both sides of \eqref{eq:KMS_pres} agree as they are zero -- recall that as $\alpha. \beta \in I_A$, in that case we have $\alpha_{|\alpha|}\neq \beta_{|\beta|-1}$. Now, if $\alpha={\o}$ and $|\beta|>1$ we get $\alpha'={\o}$ and $\alpha'' \neq {\o}$, thus both sides of \eqref{eq:KMS_pres} again agree as they are zero. Now, for $\alpha={\o}$ and $|\beta|=1$, we focus on $C_c^{\infty}(\Sigma_A)$, i.e.\ the $^*$-algebra generated by $\{S_{\beta \nu} S_{\beta \nu}^*: \beta \in V_A^1,\nu \in V_A, \beta \nu \in V_A\}$ where we have
\begin{equation}\label{eq:phi_A_1}
\varphi_A(S_{\beta \nu}S_{\beta \nu}^*)=\sum_{\nu',\delta} u_{\beta \nu,\delta\nu'}(u_{\beta \nu, \delta \nu'})^*\otimes S_{\delta \nu'}S_{\delta \nu'}^*,
\end{equation}
with $\nu', \delta$ still as above.

In this case $(\id\otimes \tau) \varphi_A (S_{\beta \nu}S_{\beta \nu}^*)= \tau(S_{\beta \nu}S_{\beta \nu}^*)1$, due to Proposition \ref{prop:equal_volume}. This completes the proof.
\end{proof}

We are now ready for the first part of our result.

\begin{prop}\label{prop:G_A_isom}
We have $(\GAi,\varphi_A)\in Q_A(D).$
\end{prop}

\begin{proof}
The fact that $\varphi_A$ preserves $\tau$ is proved in Proposition \ref{prop:actions} and by Lemma \ref{lem:preserveKMS_2} we have 
\begin{equation*}
U_{\varphi_A}(\CAi \odot \rho(\LC)\xi)=\CAi \odot \rho(\LC)\xi.
\end{equation*}

To conclude, it is enough to show that $U_{\varphi_A}$ commutes with $1\otimes D$ on $\CAi \odot \rho(\LC)\xi$.

Indeed, as $\rho(\LC)\xi$ is a core for $D$, we have that $\CAi \odot \rho(\LC)\xi$ is a core for $1\otimes D$. Pick $\eta \in \Dom(1\otimes D)$ and a sequence $\eta_n\in \CAi \odot \rho(\LC)\xi$ converging to $\eta$ in $\Dom(1\otimes D)$; namely $\eta_n \to \eta,\,\, (1\otimes D)(\eta_n)\to (1\otimes D)(\eta)$ inside $\CAi \otimes L^2(O_A,\tau)$. Then, $U_{\varphi_A}(\eta_n)\to U_{\varphi_A}(\eta)$ and $$(1\otimes D)U_{\varphi_A}(\eta_n)= U_{\varphi_A}(1\otimes D)(\eta_n)\to U_{\varphi_A}(1\otimes D)(\eta).$$ As $(1\otimes D)$ is closed, we get $U_{\varphi_A}(\eta)\in \Dom (1\otimes D)$ and $(1\otimes D)U_{\varphi_A}(\eta)= U_{\varphi_A}(1\otimes D)(\eta)$. Doing the same for $U_{\varphi_A}^*$ we obtain.
\begin{equation}\label{eq:G_A_isom_2}
U_{\varphi_A}(\Dom (1\otimes D))= \Dom (1\otimes D),\qquad U_{\varphi_A}(1\otimes D)U_{\varphi_A}^*=1\otimes D.
\end{equation}
We now prove the aforementioned commutation relation. First recall that $D=-\Delta + V$, where $\Delta$ is the log-Laplacian and $V=(2P_A-1)M_{L}$ is the potential, see \eqref{eq:D}. We will show an a priori stronger fact, that $U_{\varphi_A}$ commutes with $1\otimes V$ and $1\otimes \Delta$, separately. Also, in the following computations we switch from working on $L^2(O_A,\tau)$ to $L^2(\Gamma_A,\mu_{\Gamma_A})$. 

\vspace{0.2cm}

\noindent \textbf{Commuting with $1\otimes V$:} The collection $\chi_{\Gamma_{\alpha.\beta}}$ with $\alpha.\beta \in \tilde{V}_A$ (see \eqref{eq:Fock}) spans a dense subset of $F_A=\text{Im}(P_A)$. Following Remark \ref{rem: char_functions}, every such $\chi_{\Gamma_{\alpha.\beta}}$ equals $\chi_{\alpha \beta}\star \chi_{\beta}^*$, and so \eqref{eq:phi_A} gives $U_{\varphi}(\CAi \otimes F_A) \subset \CAi \otimes F_A$ and similarly for $U_{\varphi}^*$, hence 
$$U_{\varphi}(\CAi \otimes F_A) = \CAi \otimes F_A.$$ 
This means that $U_{\varphi}$ commutes with $1\otimes P_A$. Moreover, from \eqref{eq:phi_A} one immediately sees that $U_{\varphi_A}$ commutes with $1\otimes M_{L}$. As a result, $U_{\varphi_A}$ commutes with $1\otimes V$. 

\vspace{0.2cm}

\noindent \textbf{Commuting with $1\otimes \Delta$:} It suffices to evaluate the commutation relation on functions $\chi_{\alpha \beta_{|\beta|}\nu} \star \chi_{\beta \nu}^*=\chi_{s^{-1}_{\alpha.\beta} (C(\beta \nu))}$ for $\alpha.\beta \in I_A$ and $\nu\in V_A$ such that $\beta \nu \in V_A$. Also, recall the eigenvalues $\lambda^A$ of $\Delta$ listed in \eqref{eq:eigenvalues}. Then, we have
\begin{equation}\label{eq:G_A_isom_3}
\Delta(\chi_{\alpha \beta_{|\beta|}\nu} \star \chi_{\beta \nu}^*)=\sum_{j}c_{\beta,\nu,j} \chi_{\alpha \beta_{|\beta|}j} \star \chi_{\beta j}^*,
\end{equation}
where the sum is over all $j\in V_A$ with $|j|=|\nu|,\,\, \beta j\in V_A$, and
$$c_{\beta,\nu,j}=
\begin{cases}
\lambda_{\beta}^A(\beta \nu)-\lambda_{\max}u_{(\beta \nu)_{|\beta \nu|}}, &\text{if } j=\nu\\
-|\beta j\wedge \beta \nu|_{\lambda}^{-\df}\mu(C(\beta \nu)), &\text{if } j\neq \nu 
\end{cases},\qquad \text{where } |\cdot|_{\lambda}:=\lambda^{-|\cdot|}.$$
Indeed, one has $\Delta(\chi_{\alpha \beta_{|\beta|}\nu}\star \chi_{\beta \nu}^*)=\Delta_{\alpha.\beta}(\chi_{\alpha \beta_{|\beta|}\nu}\star \chi_{\beta \nu}^*)$. Let $(x,n,y)\in \Gamma_{\alpha.\beta}$ with $y\in C(\beta \nu)$ and for $0\leq n\leq |\nu|-1$ define the annulus $B_{y,n}:=B(y,\lambda^{-|\beta|-n})\setminus B(y,\lambda^{-|\beta|-n-1})$. Then, 
\small
\begin{align*}
\Delta_{\alpha.\beta}(\chi_{\alpha \beta_{|\beta|}\nu}\star \chi_{\beta \nu}^*)(x,n,y)&=\int_{\Gamma_{\alpha.\beta}}\frac{\chi_{\alpha \beta_{|\beta|}\nu}\star \chi_{\beta \nu}^*(x,n,y)-\chi_{\alpha \beta_{|\beta|}\nu}\star \chi_{\beta \nu}^*(z,m,w)}{\d\left((x,n,y),(z,m,w)\right)^{\df}} \d \mu_{\Gamma_A}(z,m,w)\\
&=\int_{\Gamma_{\alpha.\beta}\setminus s^{-1}_{\alpha.\beta} (C(\beta \nu))} \frac{1}{\d\left((x,n,y),(z,m,w)\right)^{\df}}\d \mu_{\Gamma_A}(z,m,w)\\
&=\int_{C(\beta)\setminus C(\beta \nu)}\frac{1}{\d(y,w)^{\df}}\d \mu (w)\\
&=\sum_{n=0}^{|\nu|-1} \int_{B_{y,n}}\frac{1}{\d(y,w)^{\df}}\d \mu (w)\\
&=\sum_{n=0}^{|\nu|-1} \lambda^{(|\beta|+n)\df}\left(\mu(B(y,\lambda^{-|\beta|-n}))-\mu(B(y,\lambda^{-|\beta|-n-1}))\right)\\
&=\lambda_{\beta}^A(\beta \nu)-\lambda_{\max}u_{(\beta \nu)_{|\beta \nu|}}.
\end{align*}
\normalsize
If now $y\in C(\beta)\setminus C(\beta \nu)$, then 
\begin{align*}
\Delta_{\alpha.\beta}(\chi_{\alpha \beta_{|\beta|}\nu}\star \chi_{\beta \nu}^*)(x,n,y)&=-\int_{C(\beta \nu)}\frac{1}{\d(y,w)^{\df}}\d \mu (w)\\
&=-\d(y,C(\beta \nu))^{-\df}\mu(C(\beta \nu))
\end{align*}
and Equation \eqref{eq:G_A_isom_3} follows. 

Our goal now is to show that 
\begin{equation}\label{eq:G_A_isom_4}
U_{\varphi_A}(1\otimes \Delta)U_{\varphi_A}^*(1\otimes \chi_{\alpha \beta_{|\beta|}\nu}\star \chi_{\beta \nu}^*)=(1\otimes \Delta)(1\otimes \chi_{\alpha \beta_{|\beta|}\nu}\star \chi_{\beta \nu}^*).
\end{equation}
From \eqref{eq:phi_A} we have 
\begin{align*}
\Sigma_1&:=U_{\varphi_A}^*(1\otimes \chi_{\alpha \beta_{|\beta|}\nu}\star \chi_{\beta \nu}^*)=\sum_{J_1} k_{A}(u_{\alpha\beta_{|\beta|}\nu,\alpha'\delta\nu'}(u_{\beta \nu, \alpha'' \delta \nu'})^*)\otimes \chi_{\alpha'\delta \nu'}\star \chi_{\alpha''\delta \nu'}^*\\
&=\sum_{J_1}u_{\alpha'' \delta \nu',\beta \nu}(u_{\alpha'\delta\nu',\alpha\beta_{|\beta|}\nu})^*\otimes \chi_{\alpha'\delta \nu'}\star \chi_{\alpha''\delta \nu'}^*,
\end{align*}
where the index set $J_1$ is the collection of $\alpha',\alpha'',\nu',\delta\in V_A$ such that 
\begin{itemize}
\item $\alpha'\delta\nu', \alpha''\delta\in V_A$;
\item $|\alpha'|=|\alpha|,|\nu'|=|\nu|, |\delta|=|\beta_{|\beta|}|=1,|\alpha''|=|\beta|-1,\alpha'_{|\alpha'|}\neq \alpha''_{|\alpha''|}$.
\end{itemize}
Then, $$\Sigma_2:=(1\otimes \Delta)(\Sigma_1)=\sum_{J_2}c_{\alpha'' \delta,\nu',j}  u_{\alpha'' \delta \nu',\beta \nu}(u_{\alpha'\delta\nu',\alpha\beta_{|\beta|}\nu})^*\otimes \chi_{\alpha'\delta j}\star \chi_{\alpha''\delta j}^*,$$ where $J_2$ is the collection of $\alpha',\alpha'',\nu',\delta,j\in V_A$ such that 
\begin{itemize}
\item $\alpha'\delta\nu', \alpha''\delta j\in V_A$;
\item $|\alpha'|=|\alpha|,|j|=|\nu'|=|\nu|, |\delta|=|\beta_{|\beta|}|=1,|\alpha''|=|\beta|-1,\alpha'_{|\alpha'|}\neq \alpha''_{|\alpha''|}$.
\end{itemize} 
Moreover, 
\begin{align*}
\Sigma_3&:=U_{\varphi_A}(\Sigma_2)= \sum_{J_2}U_{\varphi_A}(1\otimes  \chi_{\alpha'\delta j}\star \chi_{\alpha''\delta j}^*)\cdot c_{\alpha'' \delta,\nu',j}  u_{\alpha'' \delta \nu',\beta \nu}(u_{\alpha'\delta\nu',\alpha\beta_{|\beta|}\nu})^*\\
&=\sum_{J_3}c_{\alpha'' \delta,\nu',j}u_{\alpha'\delta j, \zeta\delta'j'}(u_{\alpha''\delta j,\zeta' \delta' j'})^*u_{\alpha'' \delta \nu',\beta \nu}(u_{\alpha'\delta\nu',\alpha\beta_{|\beta|}\nu})^*\otimes \chi_{\zeta \delta' j'}\star \chi_{\zeta' \delta' j'}^*,
\end{align*}
where $J_3$ is the collection of $\alpha',\alpha'',\zeta,\zeta',\nu',\delta,\delta', j,j'\in V_A$ such that 
\begin{itemize}
\item $\alpha'\delta\nu',\alpha''\delta j,\zeta\delta'j',\zeta'\delta'j'\in V_A$;
\item $|\zeta|=|\alpha'|=|\alpha|, |\delta'|=|\delta|=|\beta_{|\beta|}|=1, |j'|=|j|=|\nu'|=|\nu|$;
\item $|\zeta'|=|\alpha''|=|\beta|-1,\alpha'_{|\alpha'|}\neq \alpha''_{|\alpha''|}, \zeta_{|\zeta|}\neq \zeta'_{|\zeta'|}$.
\end{itemize} 
In particular, 
$$
\Sigma_3=\sum_{J_3}c_{\alpha'' \delta,\nu',j}u_{\alpha',\zeta}p_{\delta j, \delta' j'}(u_{\alpha'',\zeta'})^*u_{\alpha'',\hat{\beta}}p_{\delta \nu',\beta_{|\beta|}\nu}(u_{\alpha',\alpha})^*\otimes \chi_{\zeta \delta' j'}\star \chi_{\zeta' \delta' j'}^*,
$$
where we make the convention that $u_{{\o},{\o}}=1$ (this is relevant only when $|\beta|=1$, in which case $\zeta'=\alpha''={\o}$). In fact, due to orthogonality we can reduce $J_3$ by assuming $\zeta'=\hat{\beta}$ (where $\hat{\beta}={\o}$ if $|\beta|=1$). Hence we have $\zeta_{|\zeta|}\neq \hat{\beta}_{|\hat{\beta}|}$ and
$$\Sigma_3=\sum_{J_3'}c_{\alpha'' \delta,\nu',j}u_{\alpha',\zeta}p_{\delta j, \delta' j'}q_{\alpha'',\hat{\beta}}p_{\delta \nu',\beta_{|\beta|}\nu}(u_{\alpha',\alpha})^*\otimes \chi_{\zeta \delta' j'}\star \chi_{\hat{\beta} \delta' j'}^*.
$$
Since $\alpha''\delta \nu', \hat{\beta}\beta_{|\beta|}\nu\in V_A$, from Definition \ref{defn: Ariadne_0} (III) we have that $q_{\alpha'',\hat{\beta}}$ commutes with $p_{\delta \nu',\beta_{|\beta|}\nu}$. Then, we reduce $J_3'$ further by assuming $\delta'=\beta_{|\beta|}$, since $p_{\delta j, \delta' j'}p_{\delta \nu',\beta_{|\beta|}\nu}=0$ if $\delta'\neq \beta_{|\beta|}$. Then, 
$$\Sigma_3=\sum_{J_3''}c_{\alpha'' \delta,\nu',j}u_{\alpha',\zeta}p_{\delta j, \beta_{|\beta|} j'}q_{\alpha'',\hat{\beta}}p_{\delta \nu',\beta_{|\beta|}\nu}(u_{\alpha',\alpha})^*\otimes \chi_{\zeta \beta_{|\beta|} j'}\star \chi_{\beta j'}^*,
$$
where now $J_3''$ is the collection of $\alpha',\alpha'',\zeta,\nu',\delta, j,j'\in V_A$ such that 
\begin{itemize}
	\item $\alpha'\delta\nu',\alpha''\delta j\in V_A$;
	\item $|\zeta|=|\alpha'|=|\alpha|, =|\delta|=|\beta_{|\beta|}|=1, |j'|=|j|=|\nu'|=|\nu|$;
	\item $|\alpha''|=|\beta|-1,\alpha'_{|\alpha'|}\neq \alpha''_{|\alpha''|}$.
\end{itemize} 

Now the focus is on the coefficients appearing in \eqref{eq:G_A_isom_3}:

$$c_{\alpha'' \delta,\nu',j}=
\begin{cases}
\lambda_{\alpha'' \delta}^A(\alpha'' \delta \nu')-\lambda_{\max}u_{(\alpha''\delta \nu')_{|\alpha''\delta \nu'|}}, &\text{if } j=\nu',\\
-|\alpha'' \delta j\wedge \alpha'' \delta \nu'|_{\lambda}^{-\df}\mu(C(\alpha'' \delta \nu')), &\text{if } j\neq \nu'.
\end{cases}$$
First, decompose $\Sigma_3$ into restricted sums $\Sigma_{3}[j=\nu']$ and $\Sigma_{3}[j\neq \nu']$. Specifically, 
$$\Sigma_{3}[j=\nu']= \sum_{\alpha', \alpha'',\delta,\zeta, \nu'} c_{\alpha'' \delta,\nu',\nu'} u_{\alpha',\zeta}p_{\delta \nu', \beta_{|\beta|} \nu}q_{\alpha'',\hat{\beta}}p_{\delta \nu',\beta_{|\beta|}\nu}(u_{\alpha',\alpha})^*\otimes \chi_{\zeta \beta_{|\beta|} \nu}\star \chi_{\beta \nu}^*,$$ since for $j'\neq \nu$, $p_{\delta j, \beta_{|\beta|}j'}p_{\delta \nu',\beta_{|\beta|}\nu}=0$. The non-zero terms in $\Sigma_{3}[j=\nu']$ should have $q_{\alpha'',\hat{\beta}}p_{\delta \nu',\beta_{|\beta|}\nu}\neq 0$, and in fact $u_{\alpha'' \delta \nu', \beta \nu}\neq 0$. From Proposition \ref{prop:equal_volume}, for every $1\leq m_1\leq m_2\leq |\beta \nu|$ we then get 
\begin{equation}\label{eq:G_A_isom_5}
\mu(C\left((\alpha'' \delta \nu')_{m_1}\ldots (\alpha'' \delta \nu')_{m_2}\right)=\mu(C((\beta \nu)_{m_1}\ldots (\beta \nu)_{m_2}).
\end{equation}
As a result, the coefficient $c_{\alpha'' \delta,\nu',\nu'}$ can be replaced by $c_{\beta,\nu,\nu}.$ Similarly, 
$$\Sigma_{3}[j\neq \nu']= \sum_{\alpha'',\delta,j\neq \nu'} \sum_{\alpha',\zeta,j'\neq \nu} c_{\alpha'' \delta,\nu',j} u_{\alpha',\zeta}p_{\delta j, \beta_{|\beta|} j'}q_{\alpha'',\hat{\beta}}p_{\delta \nu',\beta_{|\beta|}\nu}(u_{\alpha',\alpha})^*\otimes \chi_{\zeta \beta_{|\beta|} j'}\star \chi_{\beta j'}^*,$$ since for $j'= \nu$ we have $p_{\delta j, \beta_{|\beta|} \nu}p_{\delta \nu',\beta_{|\beta|}\nu}=0$. For any non-zero term in $\Sigma_{3}[j\neq \nu']$ there is a $j'\neq \nu'$ such that $p_{\delta j, \beta_{|\beta|} j'}q_{\alpha'',\hat{\beta}}p_{\delta \nu',\beta_{|\beta|}\nu}\neq 0$. From \eqref{eq:G_A_isom_5} we get $\mu(C(\alpha''\delta \nu'))=\mu(C(\beta \nu)).$ Also, we get that $p_{\alpha'' \delta j,\beta j'} p_{\alpha'' \delta \nu', \beta \nu}\neq 0$, and so from Lemma \ref{lem: words_of_s} part (2) we must have $|\alpha'' \delta j \wedge \alpha'' \delta \nu'| = |\beta j' \wedge \beta \nu|.$ Consequently, $$-|\alpha'' \delta j\wedge \alpha'' \delta \nu'|_{\lambda}^{-\df}\mu(C(\alpha'' \delta \nu'))=-|\beta j'\wedge \beta \nu|_{\lambda}^{-\df}\mu(C(\beta \nu)).$$ In other words, in $\Sigma_3$ we can replace $c_{\alpha'' \delta,\nu',j}$ with $c_{\beta,\nu,j'}$.  

Therefore, 
\begin{align*}
\Sigma_3&=\sum_{J_3''}c_{\beta,\nu,j'}u_{\alpha',\zeta}p_{\delta j, \beta_{|\beta|} j'}q_{\alpha'',\hat{\beta}}p_{\delta \nu',\beta_{|\beta|}\nu}(u_{\alpha',\alpha})^*\otimes \chi_{\zeta \beta_{|\beta|} j'}\star \chi_{\beta j'}^*\\
&=\sum_{\alpha',\alpha'',\zeta,\delta, j'}c_{\beta,\nu,j'}u_{\alpha',\zeta}p_{\delta, \beta_{|\beta|}}q_{\alpha'',\hat{\beta}}p_{\delta,\beta_{|\beta|}}(u_{\alpha',\alpha})^*\otimes \chi_{\zeta \beta_{|\beta|} j'}\star \chi_{\beta j'}^*\\
&=\sum_{\alpha',\alpha'',\zeta,\delta, j'}c_{\beta,\nu,j'}u_{\alpha',\zeta}p_{\delta, \beta_{|\beta|}}q_{\alpha'',\hat{\beta}}(u_{\alpha',\alpha})^*\otimes \chi_{\zeta \beta_{|\beta|} j'}\star \chi_{\beta j'}^*\\
&=\sum_{\alpha',\alpha'',\zeta, j'} c_{\beta,\nu,j'}u_{\alpha',\zeta}q_{\alpha'',\hat{\beta}}(u_{\alpha',\alpha})^*\otimes \chi_{\zeta \beta_{|\beta|} j'}\star \chi_{\beta j'}^*\\
&= \sum_{\alpha',\zeta, j'} c_{\beta,\nu,j'}u_{\alpha',\zeta}\left[1-\sum_{\eta:\eta_{|\eta|}=\alpha'_{|\alpha'|}} q_{\eta,\hat{\beta}}\right](u_{\alpha',\alpha})^*\otimes \chi_{\zeta \beta_{|\beta|} j'}\star \chi_{\beta j'}^*\\
&= \sum_{\alpha',\zeta, j'} c_{\beta,\nu,j'}u_{\alpha',\zeta}(u_{\alpha',\alpha})^*\otimes \chi_{\zeta \beta_{|\beta|} j'}\star \chi_{\beta j'}^*\\
&=\sum_{j'}c_{\beta,\nu,j'}\otimes \chi_{\alpha \beta_{|\beta|} j'}\star \chi_{\beta j'}^*\\
&=(1\otimes \Delta)(1\otimes \chi_{\alpha \beta_{|\beta|}\nu}\star \chi_{\beta \nu}^*).
\end{align*}
This completes the proof.
\end{proof}

\subsection{Factorisation of isometric actions}

We will now prove that conversely, any $D$-isometric action of a compact quantum group on $O_{A}$ factors through $\varphi_{A}$. We start with an elementary Hilbert space result for which we provide a proof for completeness (it can be also deduced from Halmos' description of projections in general position).

\begin{lemma}\label{lem:FP_theorem}
Let $v,w$ be projections acting on a complex Hilbert space $H$. Then, $vw=wv$ if and only if $vwv=wvwv$.
\end{lemma}

\begin{proof}
One direction is obvious. For the other one, assume that $vwv=wvwv$ and observe that $P = vwv$ is a projection: $P^{2} = vwvwv = v(vwv) = P$ and that $r = wv$ satisfies $r^{*}r = P$. In other words, $r$ is a partial isometry. As a consequence, $rr^{*} = wvw$ is a projection, from which we deduce that $$wvw = (wvw)^{2} = wvwvw = vwvw = (wvwv)^{*} = P.$$
As a consequence, $vrr^{*} = vP = P = rr^{*}$ and we conclude that $vwv = vr = vrr^{*}r = rr^{*}r = r$. In other words, $r = vwv$ is self-adjoint and $vw = r^{*} = r = wv$.


\end{proof}

\begin{thm}\label{thm:factoring}
The object $(\GAi,\varphi_A)$ is final in $Q_A(D)$. Namely, for every object $(\mathbb G,\varphi)$ in $Q_A(D)$ there is a unique compact quantum group homomorphism $\pi:\CAi \to \CG$ such that the following diagram commutes:
$$\begin{tikzcd}
  & \qquad \CAi \otimes O_A \arrow[dr,"\pi\otimes 1"] \\
O_A \arrow[ur,"\varphi_A"] \arrow[rr,"\varphi"] &&  \CG \otimes O_A
\end{tikzcd}.$$ 
If the action $\varphi$ is faithful then $\pi$ is surjective.
\end{thm}

\begin{proof}
Fix an object $(\mathbb G,\varphi)$ in $Q_A(D)$.

\smallskip

\noindent \textbf{Preserving $E_1$:} We have that $U_{\varphi}$ preserves $\CG \otimes E_1$, where $E_1\subset L^2(O_A,\tau)$ is the finite dimensional subspace $$E_1=\langle\rho(S_{\beta}S_{\beta}^*)\xi: \beta \in V_A^1 \rangle.$$ Therefore, for every $\beta \in V_A^1$ we have
\begin{equation}\label{eq:3}
U_{\varphi}(1\otimes \rho(S_{\beta}S_{\beta}^*)\xi)=\sum_{\beta'} r_{\beta,\beta'} \otimes \rho(S_{\beta'}S_{\beta'}^*)\xi.
\end{equation}
Since $1\otimes D$ is selfadjoint it also commutes with $U^*_{\varphi}$, thus 
\begin{equation}\label{eq:4}
U_{\varphi}^*(1\otimes \rho(S_{\beta}S_{\beta}^*)\xi)=\sum_{\beta'} w_{\beta,\beta'} \otimes \rho(S_{\beta'}S_{\beta'}^*)\xi.
\end{equation}
In fact, one has $$w_{\beta,\beta'}= (r_{\beta',\beta})^*\frac{\tau(S_{\beta}S_{\beta}^*)}{\tau(S_{\beta'}S_{\beta'}^*)},$$ and from Lemma \ref{lem:preserveKMS_2} we also get $k^{-1}(r_{\beta,\beta'})=w_{\beta,\beta'}$.

\begin{remark} 
At first glance $r=(r_{\beta,\beta'})$ and $w=(w_{\beta,\beta'})$ are not necessarily unitary in $M_N(\CG)$ since the vectors $1\otimes \rho(S_{\beta'}S_{\beta'}^*)\xi \in C(\mathbb G)\otimes L^2(O_A,\tau)$ are not normalised. Instead, the matrices $\tilde{r}=(\tilde{r}_{\beta,\beta'})$ and $\tilde{w}=(\tilde{w}_{\beta,\beta'})$ where 
\begin{equation*}
\tilde{r}_{\beta,\beta'}:=r_{\beta,\beta'}\left(\frac{\tau(S_{\beta'}S_{\beta'}^*)}{\tau(S_{\beta}S_{\beta}^*)}\right)^{1/2}\qquad \text{and} \qquad \tilde{w}_{\beta,\beta'}:=w_{\beta,\beta'}\left(\frac{\tau(S_{\beta'}S_{\beta'}^*)}{\tau(S_{\beta}S_{\beta}^*)}\right)^{1/2}
\end{equation*}
are unitary and $\tilde{w}= \tilde{r}^*$. We will see though that $r,w$ are always unitary, as the `non-matching terms' must vanish (see also Proposition \ref{prop:equal_volume}).
\end{remark}

As a result, we obtain that 
\begin{equation}\label{eq:5}
\varphi(S_{\beta}S_{\beta}^*)=\sum_{\beta'=1}^N r_{\beta, \beta'}\otimes S_{\beta'}S_{\beta'}^*, \qquad
(k^{-1}\otimes 1)\varphi(S_{\beta}S_{\beta}^*)=\sum_{\beta'=1}^N w_{\beta, \beta'}\otimes S_{\beta'}S_{\beta'}^*
\end{equation}
As $\varphi$ is a $*$-homomorphism and $k^{-1}$ respects idempotents, the projections $S_{\beta'}S_{\beta'}^*$, $\beta'=1, \ldots,N$ are mutually orthogonal and $\sum_{\beta=1}^N\varphi(S_{\beta}S_{\beta}^*)=1\otimes 1,$ using \eqref{eq:5} we obtain that for $1\leq \alpha,\beta,\delta \leq N$,
\begin{enumerate}[(i)]
\item each $r_{\alpha, \beta}$ is a projection in $\CG$;
\item each $w_{\alpha, \beta}$ is a projection in $\CG$. In particular,
$$\text{if } \tau(S_{\alpha}S^*_{\alpha})\neq \tau(S_{\beta}S^*_{\beta})\,\,\, \text{then } r_{\alpha,\beta}=0.$$
\item $\sum_{\alpha}r_{\alpha, \beta}=\sum_{\beta}r_{\alpha, \beta}=1$. Hence, $r_{\alpha, \beta}r_{\delta, \beta}=0$ and $r_{\beta,\alpha}r_{\beta,\delta}=0$ if $\alpha \neq \delta$.
\end{enumerate}
Hence, the matrix $r=(r_{\alpha,\beta})_{1\leq \alpha,\beta \leq N}$ is a magic unitary and its adjoint is $w=(w_{\alpha,\beta})_{1\leq \alpha,\beta \leq N}$.

\smallskip

\noindent \textbf{Preserving $E_2$:} In the same spirit, the unitary $U_{\varphi}$ preserves $\CG \otimes E_2$, where $E_2\subset L^2(O_A,\tau)$ is the finite dimensional subspace $$E_2=\langle\rho(S_{\alpha}S_{\beta}S_{\beta}^*)\xi: \alpha, \beta \in V_A^1 \,\,\, \text{and } \alpha.\beta \in I_A\rangle.$$ Therefore, for every $\alpha, \beta \in V_A^1$ with $\alpha.\beta \in I_A$,
\begin{equation}\label{eq:6}
U_{\varphi}(1\otimes \rho(S_{\alpha}S_{\beta}S_{\beta}^*)\xi)=\sum_{\alpha'.\beta'} y_{\alpha.\beta,\alpha'.\beta'} \otimes \rho(S_{\alpha'}S_{\beta'}S_{\beta'}^*)\xi,
\end{equation}
where the sum is over all $\alpha'.\beta'\in I_A$. Similarly, we have 
\begin{equation}\label{eq:7}
U_{\varphi}^*(1\otimes \rho(S_{\alpha}S_{\beta}S_{\beta}^*)\xi)=\sum_{\alpha'.\beta'} z_{\alpha.\beta,\alpha'.\beta'} \otimes \rho(S_{\alpha'}S_{\beta'}S_{\beta'}^*)\xi
\end{equation}
and hence
\begin{equation}\label{eq:8}
z_{\alpha.\beta,\alpha'.\beta'}= (y_{\alpha'.\beta',\alpha.\beta})^*\frac{\tau(S_{\beta}S_{\beta}^*)}{\tau(S_{\beta'}S_{\beta'}^*)}\qquad \text{and}\qquad k^{-1}(y_{\alpha.\beta,\alpha'.\beta'})=z_{\alpha.\beta,\alpha'.\beta'}.
\end{equation}
Now, since for $\alpha, \beta \in V_A^1$ with $\alpha.\beta \not\in I_A$, $S_{\alpha}S_{\beta}S_{\beta}^*=0$, summing over all $1\leq \beta \leq N$, for every $1\leq \alpha \leq N$ we obtain
\begin{equation}\label{eq:9}
\varphi(S_{\alpha})=\sum_{\alpha'.\beta'} s_{\alpha' . \beta'}^{\alpha} \otimes S_{\alpha'}S_{\beta'}S_{\beta'}^*\qquad  \text{and} \qquad (k^{-1}\otimes 1)\varphi(S_{\alpha})=\sum_{\alpha'.\beta'} t_{\alpha' . \beta'}^{\alpha} \otimes S_{\alpha'}S_{\beta'}S_{\beta'}^*.
\end{equation}
This can be extended for $\alpha\in V_A\setminus \{{\o}\}$, by summing over all $\alpha'.\beta' \in I_A$ where $|\alpha'|=|\alpha|$, $\beta' \in V_A^{1}$ and $s_{\alpha' . \beta'}^{\alpha}= s_{\alpha_1'.\alpha_2'}^{\alpha_1}s_{\alpha_2'.\alpha_3'}^{\alpha_2}\ldots s_{\alpha'_{|\alpha'|}.\beta'}^{\alpha_{|\alpha|}}$ as well as $t_{\alpha' . \beta'}^{\alpha}=t_{\alpha'_{|\alpha'|}.\beta'}^{\alpha_{|\alpha|}}\ldots t_{\alpha_2'.\alpha_3'}^{\alpha_2}t_{\alpha_1'.\alpha_2'}^{\alpha_1}$, namely
\begin{equation}\label{eq:10}
\varphi(S_{\alpha})=\sum_{\alpha'.\beta'} s_{\alpha' . \beta'}^{\alpha} \otimes S_{\alpha'}S_{\beta'}S_{\beta'}^* \qquad \text{and} \qquad (k^{-1}\otimes 1)\varphi(S_{\alpha})=\sum_{\alpha'.\beta'} t_{\alpha' . \beta'}^{\alpha} \otimes S_{\alpha'}S_{\beta'}S_{\beta'}^*.
\end{equation}

\smallskip

\noindent \textbf{Linearity of co-action:} From \eqref{eq:10} and the fact that $\varphi$ is a $*$-homomorphism, it follows that $U_{\varphi}$ preserves $$E:=\bigvee_{m,n\geq 1}\langle \CG \otimes \rho(S_{\alpha}S^*_{\beta})\xi: \alpha\in V_A^m, \beta\in V_A^n,\alpha_{|\alpha|}=\beta_{|\beta|}\rangle, $$ where $E=\bigvee_{m,n\geq 1}\langle \CG \otimes \rho(S_{\alpha}S^*_{\beta})\xi: \alpha\in V_A^m, \beta\in V_A^n\rangle$ is a complemented right $\CG$-submodule of $\CG\otimes L^2(O_A,\tau)$. Similarly for $U_{\varphi}^*$. As a result, $U_{\varphi}$ preserves also the orthogonal complement of that subspace, which is $E_+\oplus E_-$, where 
\begin{align*}
E_+&:=\langle \CG \otimes \rho(S_{\alpha})\xi: \alpha\in V_A^m, \,\, m\geq 1\rangle\\
E_-&:= \langle \CG \otimes \rho(S^*_{\beta})\xi: \beta\in V_A^m,\,\, m\geq 1\rangle.
\end{align*}
Now we get that $\varphi$ should preserve $E_+$ and $E_-$ respectively, since from \eqref{eq:6} we see that $U_{\varphi}(E_{\pm})$ and $U_{\varphi}^*(E_{\pm})$ are orthogonal to $E_{\mp}$. Similarly, from orthogonality and \eqref{eq:6} we see that $U_{\varphi}$ and $U_{\varphi}^{*}$ preserve each degree $m$. Consequently, from \eqref{eq:9} we get that the coefficients $s_{\alpha'.\beta'}^{\alpha}$ for $\alpha',\beta'\in V_A^1$ with $\alpha'.\beta'\in I_A$ are independent of $\beta'$. In other words, for every $1\leq \alpha \leq N$ we obtain
\begin{equation}\label{eq:15}
\varphi(S_{\alpha})=\sum_{\alpha'} s_{\alpha,\alpha'} \otimes S_{\alpha'}\qquad \text{and} \qquad (k^{-1}\otimes 1)\varphi(S_{\alpha})=\sum_{\alpha'} t_{\alpha, \alpha'} \otimes S_{\alpha'}.
\end{equation} 

Since $(\Phi \otimes 1) \circ \varphi = (1\otimes \varphi) \circ \varphi$, for every $1\leq \alpha,\beta \leq N$ we then get 
$$\Phi(s_{\alpha,\beta})=\sum_{\delta=1}^Ns_{\alpha,\delta}\otimes s_{\delta,\beta}.$$ 
Also, from \eqref{eq:5} and \eqref{eq:15} we see that for every $1\leq \alpha, \beta \leq N$ one has 
\begin{equation}\label{eq:16}
s_{\alpha,\beta}(s_{\alpha,\beta})^*=r_{\alpha,\beta}.
\end{equation}
In particular, each $s_{\alpha,\beta}$ is a partial isometry. Also, if $\beta \neq \delta$ then 
\begin{equation}\label{eq:17}
\langle 1\otimes \rho (S_{\alpha}S_{\alpha}^*)\xi, U_{\varphi}(1\otimes \rho (S_{\beta}S_{\delta}^*)\xi) \rangle = 0.
\end{equation}
This follows by using the equality $U_{\varphi}^*(\CG\otimes E_1)=\CG\otimes E_1$ or the cases where  $S_{\beta}S_{\delta}^*=0$. Now, expanding \eqref{eq:17} we obtain
\begin{equation}\label{eq:18}
s_{\beta,\alpha}(s_{\delta,\alpha})^*=\tau(S_{\alpha}S_{\alpha}^*)^{-1} \langle 1\otimes \rho (S_{\alpha}S_{\alpha}^*)\xi, U_{\varphi}(1\otimes \rho (S_{\beta}S_{\delta}^*)\xi) \rangle = 0.
\end{equation}

Finally, for every $1\leq \alpha \leq N$ using \eqref{eq:5} we get 
$$\varphi (S_{\alpha}^* S_{\alpha})=\sum_{\beta,\delta=1}^N A_{\alpha,\beta}s_{\beta,\delta}(s_{\beta,\delta})^*\otimes S_{\delta}S_{\delta}^*$$ and writing $\varphi(S_{\alpha}^* S_{\alpha})=\varphi(S_{\alpha})^*\varphi(S_{\alpha})$ we get 
$$\varphi(S_{\alpha}^* S_{\alpha})=\sum_{\beta,\delta=1}^N(s_{\alpha,\beta})^*s_{\alpha,\beta} A_{\beta,\delta} \otimes S_{\delta}S_{\delta}^*.$$ Then, orthogonality between the $S_{\delta}S_{\delta}^*$ viewed as vectors in $L^2(O_A,\tau)$ implies that for every $1\leq \alpha,\delta \leq N$, 
$$\sum_{\beta=1}^N A_{\alpha,\beta}s_{\beta,\delta}(s_{\beta,\delta})^*= \sum_{\beta=1}^N(s_{\alpha,\beta})^*s_{\alpha,\beta} A_{\beta,\delta}.$$

\smallskip

\noindent \textbf{Factoring through $\QG_A^{1}$-action (Properties I and II):} So far we have shown that for every $1\leq \alpha,\beta, \delta \leq N$ the elements $s_{\alpha,\beta}$ are partial isometries with range projections $r=(r_{\alpha,\beta})_{1\leq \alpha,\beta \leq N}$ and source projections $\sigma=(\sigma_{\alpha,\beta})_{1\leq \alpha,\beta \leq N}$ such that for all $\alpha, \beta, \delta = 1,\ldots,N$
\begin{enumerate}[(P1)]
\item if $\beta\neq \delta$ then $s_{\beta,\alpha}(s_{\delta,\alpha})^*=0;$
\item if $\tau(S_{\alpha}S_{\alpha}^*)\neq \tau(S_{\beta}S_{\beta}^*)$ then $s_{\alpha,\beta}=0$;
\item $\sum_{\kappa=1}^N r_{\kappa,\alpha}=\sum_{\kappa=1}^N r_{\alpha,\kappa}=1$;
\item $A$ intertwines $r$ and $\sigma$; i.e. $$\sum_{\kappa=1}^N A_{\alpha, \kappa}r_{\kappa, \delta}=\sum_{\kappa=1}^N \sigma_{\alpha, \kappa} A_{\kappa, \delta}.$$
\end{enumerate}
Property (P4) is exactly the property II from Definition \ref{defn: Ariadne_0}. For property I, observe that (P3) implies that $r$ is a magic unitary, which preserves the right Perron--Frobenius eigenvector $\vec{u}$  due to (P2). Then, due to (P4), $\sigma$ also preserves $\vec{u}$. However, showing that $\sigma$ is a magic unitary is not straightforward. To this end, we shall prove that the matrix $s= (s_{\alpha,\beta})_{1\leq \alpha,\beta \leq N}$ and its conjugate $\overline{s}=((s_{\alpha, \beta})^*)_{1\leq \alpha,\beta \leq N}$ are unitaries. 

We first claim that for every $\alpha =1, \ldots, N$ we have
\begin{equation}\label{eq:conj_unitary_1}
\sum_{\beta=1}^N \sigma_{\alpha,\beta}=1,
\end{equation}
To prove it first note that $s_{\alpha,\beta}=0$ if and only if $r_{\alpha,\beta}=0$ if and only if $\sigma_{\alpha,\beta}=0$. Now for every $1\leq \alpha,\delta \leq N$ we have 
$$\sum_{\beta=1}^N A_{\alpha,\beta} \frac{\tau(S_{\beta}S_{\beta}^*)}{\tau(S_{\delta}S_{\delta}^*)} r_{\beta,\delta}=  \sum_{\beta=1}^N A_{\alpha, \beta}r_{\beta, \delta}=\sum_{\beta=1}^N \sigma_{\alpha, \beta} A_{\beta, \delta}.$$ Multiplying both sides by $\tau(S_{\delta}S_{\delta}^*)$ gives 
$$\sum_{\beta=1}^N A_{\alpha,\beta} \tau(S_{\beta}S_{\beta}^*)r_{\beta,\delta}= \sum_{\beta=1}^N \sigma_{\alpha, \beta} A_{\beta, \delta}\tau(S_{\delta}S_{\delta}^*).$$ 
Summing over all $\delta\in \{1, \ldots, N\}$ and using (P3) gives 
$$\sum_{\beta=1}^N A_{\alpha,\beta} \tau(S_{\beta}S_{\beta}^*)=\sum_{\beta=1}^N \sigma_{\alpha, \beta} \left(\sum_{\delta} A_{\beta, \delta}\tau(S_{\delta}S_{\delta}^*)\right),$$ 
which is equivalent to 
$$\tau(S_{\alpha}^*S_{\alpha})=\sum_{\beta=1}^N \sigma_{\alpha,\beta}\tau(S_{\beta}^*S_{\beta}).$$ As a result, $$\sum_{\beta=1}^N \sigma_{\alpha,\beta} \frac{\tau(S_{\beta}^*S_{\beta})}{\tau(S_{\alpha}^*S_{\alpha})}=1.$$ 
From (P2), observe that if $\alpha, \beta \in \{1, \ldots,N\}$ and $\sigma_{\alpha,\beta}\neq 0$ then $\tau(S_{\alpha}S_{\alpha}^*)=\tau(S_{\beta}S_{\beta}^*)$, and from the KMS property of $\tau$ we have $\tau(S_{\delta}S_{\delta}^*)=\lambda_{\max}^{-1}\tau(S_{\delta}^*S_{\delta}),$ for any $1\leq \delta \leq N$. Therefore, 
$$\sum_{\beta=1}^N\sigma_{\alpha,\beta}=\sum_{\beta=1}^N \sigma_{\alpha,\beta} \frac{\tau(S_{\beta}^*S_{\beta})}{\tau(S_{\alpha}^*S_{\alpha})}=1.$$ 
In particular, if $\beta\neq \delta$ we have $\sigma_{\alpha,\beta}\sigma_{\alpha,\delta}=0$, hence 
\begin{equation}\label{eq:conj_unitary_2}
s_{\alpha,\beta}(s_{\alpha,\delta})^*=0.
\end{equation}

Now from \eqref{eq:conj_unitary_2} and (P3) we obtain $(\overline{s})^*\overline{s}=1$, hence $\overline{s}$ is unitary. Finally, since $\overline{s}$ is unitary, then $s$ is invertible. Also, from (P1) and (P3) we get that $ss^*=1$. Therefore, $s$ is unitary as well.

All these properties for $s=(s_{\alpha,\beta})_{1\leq \alpha,\beta \leq N}$ imply that the assignment $u_{\alpha,\beta}\mapsto s_{\alpha,\beta}$ extends to a $*$-homomorphism $\pi:C(\QG_A^1)\to \CG$ such that $(\pi \otimes \pi)\circ \Phi_A = \Phi \circ \pi.$ It clearly holds that $(\pi \otimes 1)\circ \varphi_A =\varphi$ and if $\pi':C(\QG_A^1) \to \CG$ is another $*$-homomorphism such that $(\pi' \otimes \pi')\circ \Phi_A = \Phi \circ \pi'$, then for every $1\leq \alpha \leq N$ we have
 $$\sum_{\beta=1}^N \pi(u_{\alpha,\beta})\otimes S_{\beta}=\sum_{\beta=1}^N \pi'(u_{\alpha,\beta})\otimes S_{\beta}.$$ Therefore, from the orthogonality of all $S_{\beta}$ viewed as vectors in $L^2(O_A,\tau)$, we conclude that $\pi=\pi'$.

\smallskip

\noindent \textbf{Factoring through $\GAi$-action (Property III):} At this point we should note that since the $\mathbb G$-action factors through the $\QG_A^1$-action, the antipode $k$ of $\mathbb G$ satisfies $$k(s_{\alpha,\beta})= (s_{\beta,\alpha})^*.$$ 
Moreover, property III is tightly connected to the underlying dynamics and hence it will be more convenient to work with the groupoid picture of $O_A$. Namely, the operators $1\otimes D, U_{\varphi},U_{\varphi}^*$ will now act on $\LC \subset L^2(\Gamma_A,\mu_{\Gamma_A}).$ First we will show that properties I and II imply that $U_{\varphi}$ commutes with $1\otimes P_A$.

Recall that the collection $\rho(S_{\alpha \beta}S_{\beta}^*)\xi=\chi_{\Gamma_{\alpha.\beta}}$ with $\alpha.\beta \in \tilde{V}_A$ (recall the notation introduced in \eqref{eq:Fock}) spans a dense subset of $F_A=\textup{Im}(P_A)$. Moreover, we have that for $\alpha, \beta$ as above
\begin{equation}\label{eq:19}
U_{\varphi}(1\otimes \rho(S_{\alpha \beta}S_{\beta}^*)\xi) = \sum_{\alpha ' ,\beta'} s_{\alpha \beta, \alpha'}(s_{\beta,\beta'})^* \otimes \rho(S_{\alpha'} S_{\beta'}^*)\xi,
\end{equation}
where the sum is over $\alpha', \beta'  \in V_A$ with $|\alpha'|=|\alpha|+1$, $|\beta'|=1$. Also, $\alpha'_{|\alpha'|}=\beta'$ since the rightmost product in the expansion of the coefficients is $s_{\beta, \alpha'_{|\alpha'|}}(s_{\beta,\beta'})^*.$ In particular, for every such $\alpha',\beta'$, $$\rho(S_{\alpha'} S_{\beta'}^*)\xi = \chi_{\Gamma_\gamma},$$ with $\gamma=\alpha_1'\ldots \alpha_{|\alpha'|-1}'.\beta' \in \widetilde{V_A}$. As a result, $U_{\varphi}(\CG \otimes F_A) \subset \CG \otimes F_A$. Since the same is true for $U_{\varphi}^*$ we get $U_{\varphi}(\CG \otimes F_A) = \CG \otimes F_A$, in other words $U_{\varphi}$ commutes with $1\otimes P_A$. From \eqref{eq:19} it also follows that $U_{\varphi}$ commutes with $1\otimes M_{L}P_A$. 

We are now ready to derive property III by induction on the length of admissible words. Namely, we know that the $\mathbb G$-action factors through the $\QG_A^1$-action. Assuming that it factors through the $\GA$-action, we will show that it factors through the $\QG_A^{\ell+1}$-action. This amounts to showing that for all $\alpha,\beta\in V_A^{1}$ and $\delta,\nu\in V_A^{\ell}$ with $\alpha \delta,\beta \nu \in V_A$ it holds that
\begin{equation}\label{eq:ind_1}
\sigma_{\alpha,\beta}r_{\delta,\nu}= r_{\delta,\nu}\sigma_{\alpha,\beta}.
\end{equation}

To derive \eqref{eq:ind_1} we proceed as follows: for $\alpha,\beta\in V_A^{1}$ and $\nu \in V_A^{\ell}$ with $\alpha \nu, \beta \nu \in V_A$ denote $f_{\alpha,\beta,\nu}:= \chi_{\alpha \nu} \star \chi_{\beta \nu}^*$. From Remark \ref{rem: char_functions}, we obtain that 
$$ f_{\alpha,\beta,\nu}=
\begin{cases}
\chi_{s^{-1}_{{\o}.\beta}(C(\beta \nu))}, &\text{if } \alpha=\beta,\\
\chi_{s^{-1}_{\alpha.\beta \nu_1}(C(\beta \nu))}, &\text{if } \alpha \neq \beta .
\end{cases}
$$
Note that if $(\alpha,\beta,\nu) \neq (\alpha',\beta',\nu')$ then $f_{\alpha, \beta, \nu}$ is orthogonal to $f_{\alpha', \beta', \nu'}$ in $L^2(\Gamma_A,\mu_{\Gamma_A})$. From the induction step we have that $U_{\varphi}$ preserves $\CG \otimes \langle f_{\alpha,\beta,\nu}\rangle,$ see \eqref{eq:orthogonality2} in Lemma \ref{lem: words_of_s}. Also, we know that $ U_{\varphi}$ commutes with $1\otimes D$ so with $1\otimes |D|$ as well, where $|D|=\Delta + M_{L}$. 

In particular, for $\beta\in V_A^{1}$ and $\nu\in V_A^{\ell}$ with $\beta \nu \in V_A$ we have 
\begin{equation}\label{eq:ind_2}
U_{\varphi}(1\otimes |D|)U_{\varphi}^*(1\otimes f_{\beta,\beta, \nu})= (1\otimes |D|)(1\otimes f_{\beta,\beta,\nu}).
\end{equation}
We shall expand \eqref{eq:ind_2} to derive new algebraic conditions. First, it is clear that
\begin{equation}\label{eq:ind_3}
M_{L}(f_{\alpha,\beta,\nu})=
\begin{cases}
f_{\alpha,\beta,\nu}, &\text{if } \alpha=\beta,\\
3f_{\alpha,\beta,\nu}, &\text{if } \alpha \neq \beta.
\end{cases}
\end{equation}
Moreover, set $\check{\nu}:=\nu_2\ldots \nu_{|\nu|}$, which is assumed to be the empty word when $|\nu|=1$. Then, from \eqref{eq:G_A_isom_3} we have that 
\begin{equation}\label{eq:ind_4}
\Delta(f_{\alpha,\beta,\nu})=
\begin{cases}
\sum_{j}c_{\beta,\nu,j}f_{\beta,\beta,j}, &\text{if } \alpha=\beta,\\
\sum_{j'}c_{\beta \nu_1,\check{\nu},j'}f_{\alpha \nu_1,\beta \nu_1,j'}, &\text{if } \alpha \neq \beta,
\end{cases}
\end{equation}
where the sums are over all $j,j' \in V_A^{|\beta|}$ such that $\beta j, \beta j' \in V_A$.

Eventually, from \eqref{eq:ind_3} and \eqref{eq:ind_4} we obtain that 
\begin{equation}\label{eq:ind_5}
(1\otimes |D|)(1\otimes f_{\beta,\beta,\nu})=(c_{\beta,\nu,\nu}+1)\otimes f_{\beta,\beta,\nu}+\sum_{j\neq \nu}c_{\beta,\nu,j}\otimes f_{\beta,\beta,j}.
\end{equation}
Further, to calculate $U_{\varphi}(1\otimes |D|)U_{\varphi}^*(1\otimes f_{\beta,\beta,\nu})$ note that $$U_{\varphi}^*(1\otimes f_{\beta,\beta,\nu})=\sum_{\beta',\beta'',\nu'}s_{\beta'',\beta}r_{\nu',\nu}(s_{\beta',\beta})^*\otimes f_{\beta',\beta'',\nu'}.$$
Then, we write $(1\otimes |D|)U_{\varphi}^*(1\otimes f_{\beta,\beta,\nu})=\Sigma_1+\Sigma_2+\Sigma_3+\Sigma_4$, where 
\begin{align*}
\Sigma_1&=3\sum_{\beta'\neq \beta'',\nu'}s_{\beta'',\beta}r_{\nu',\nu}(s_{\beta',\beta})^*\otimes f_{\beta',\beta'',\nu'}\\
\Sigma_2&=\sum_{\beta'\neq \beta'',\nu',j'}c_{\beta'' \nu_1',\check{\nu'},j'}s_{\beta'',\beta}r_{\nu',\nu}(s_{\beta',\beta})^*\otimes f_{\beta',\beta'' ,\nu_1'j'}\\
\Sigma_3&=\sum_{\beta', \nu'\neq j} c_{\beta',\nu',j}s_{\beta',\beta}r_{\nu',\nu}(s_{\beta',\beta})^*\otimes f_{\beta',\beta',j}\\
\Sigma_4&=\sum_{\beta',\nu'}(c_{\beta',\nu',\nu'}+1)s_{\beta',\beta}r_{\nu',\nu}(s_{\beta',\beta})^*\otimes f_{\beta',\beta',\nu'}.
\end{align*}
Now applying $U_{\varphi}$ we get 
\begin{align}\label{eq:ind_6}
U_{\varphi}(\Sigma_1)&=3\sum_{i',m'}\left[ \sum_{\beta'\neq \beta'',\nu'} s_{\beta',i'}r_{\nu',m'}\sigma_{\beta'',\beta}r_{\nu',\nu}(s_{\beta',\beta})^* \right]\otimes f_{i',\beta,m'}\\
\nonumber
&=3\sum_{i',m'}\left[ \sum_{\beta',\nu'} s_{\beta',i'}r_{\nu',m'}(1-\sigma_{\beta',\beta})r_{\nu',\nu}(s_{\beta',\beta})^* \right]\otimes f_{i',\beta,m'}\\
\nonumber
&=3\otimes f_{\beta,\beta,\nu} - 3\sum_{i',m'}\left[ \sum_{\beta',\nu'} s_{\beta',i'}r_{\nu',m'}\sigma_{\beta',\beta}r_{\nu',\nu}(s_{\beta',\beta})^* \right]\otimes f_{i',\beta,m'}.
\end{align}

Moreover, we obtain
\begin{equation}\label{eq:ind_7}
U_{\varphi}(\Sigma_2)=\sum_{\beta'\neq \beta'',\nu',j',i',m'}c_{\beta'' \nu_1',\check{\nu'},j'}s_{\beta',i'}r_{\nu_1'j',m'}\sigma_{\beta'',\beta}r_{\nu',\nu}(s_{\beta',\beta})^* \otimes f_{i',\beta,m'}.
\end{equation} 
The sum \eqref{eq:ind_7} is zero if $\ell=1$, since the scalar coefficients $c_{\beta''\nu_1',{\o},{\o}}$ vanish. For $\ell\geq 2$ the proof is more involved. The easiest observation is that the non-zero terms have $m_1'=\nu_1$. We can decompose $U_{\varphi}(\Sigma_2)$ into 
$$U_{\varphi}(\Sigma_2)[j'=\check{\nu'}]+U_{\varphi}(\Sigma_2)[j'\neq \check{\nu'},m'\neq \nu]+U_{\varphi}(\Sigma_2)[j'\neq \check{\nu'}, m'=\nu],$$ 
where the indices denote the restrictions. Any non-zero term in $U_{\varphi}(\Sigma_2)[j'=\check{\nu'}]$ has $\sigma_{\beta'',\beta}r_{\nu',\nu}\neq 0$. In particular, $s_{\beta'',\beta}r_{\nu',\nu}\neq 0$ and hence 
$$\mu(C(\beta'' \nu'))=\mu(C(\beta \nu)).$$ 
This is derived from $\langle U_{\varphi}(1\otimes \chi_{\nu'}\star \chi_{\beta'' \nu'}^*), 1\otimes \chi_{\nu}\star \chi_{\beta \nu}^* \rangle = \langle 1\otimes \chi_{\nu'}\star \chi_{\beta'' \nu'}, U_{\varphi}^*(1\otimes \chi_{\nu}\star \chi_{\beta \nu}^*)\rangle$, using the specific formula \eqref{eq:KMS} for the KMS state $\tau$. Moreover, since $s_{\beta'',\beta} s_{\nu_1',\nu_1}\ldots s_{\nu_n',\nu_n}$ is a non-zero partial isometry for every $1\leq n\leq \ell-1$, working as in Proposition \ref{prop:equal_volume} we get that
$$\mu(C(\beta''\nu_1'\ldots \nu_n'))=\mu(C(\beta \nu_1\ldots \nu_n)).$$ 
This means that $c_{\beta'' \nu_1',\check{\nu'},\check{\nu'}}$ can be replaced with $c_{\beta \nu_1, \check{\nu}, \check{\nu}}.$ As a result, we have 
\begin{align}\label{eq:ind_8}
U_{\varphi}(\Sigma_2)[j'=\check{\nu'}]&=\sum_{\beta'\neq \beta'',\nu',i',m'}c_{\beta \nu_1, \check{\nu}, \check{\nu}}s_{\beta',i'}r_{\nu',m'}\sigma_{\beta'',\beta}r_{\nu',\nu}(s_{\beta',\beta})^* \otimes f_{i',\beta,m'}\\
\nonumber
&=\sum_{\beta',\nu',i',m'}c_{\beta \nu_1, \check{\nu}, \check{\nu}}s_{\beta',i'}r_{\nu',m'}(1-\sigma_{\beta',\beta})r_{\nu',\nu}(s_{\beta',\beta})^* \otimes f_{i',\beta,m'}\\
\nonumber
&=c_{\beta \nu_1, \check{\nu}, \check{\nu}}\otimes f_{\beta,\beta,\nu}-\sum_{\beta',\nu',i',m'}c_{\beta \nu_1, \check{\nu}, \check{\nu}}s_{\beta',i'}r_{\nu',m'}\sigma_{\beta',\beta}r_{\nu',\nu}(s_{\beta',\beta})^* \otimes f_{i',\beta,m'}.
\end{align}
At this point, from \eqref{eq:eigenvalues} we see that 
$$c_{\beta \nu_1,\check{\nu},\check{\nu}}= c_{\beta,\nu,\nu}-\lambda_{\max}\mu(C(\beta)\setminus C(\beta \nu_1)).$$

Now for any non-zero term in $U_{\varphi}(\Sigma_2)[j'\neq \check{\nu'},m'\neq \nu]$ we have $r_{\nu_1'j',m'}\sigma_{\beta'',\beta}r_{\nu',\nu}\neq 0$ and so 
$$(s_{\beta''\nu_1'\hat{j'},\beta \hat{m'}})^*s_{\beta''\hat{\nu'},\beta \hat{\nu}}\neq 0.$$ 
As the indices have length $\ell$, from Lemma \ref{lem: words_of_s} we obtain that $|\beta''\nu_1'\hat{j'} \wedge \beta''\hat{\nu'} |= |\beta \hat{m'} \wedge \beta \hat{\nu}|,$ and given that $j'\neq \check{\nu'}, m'\neq \nu$ this means that $$|\beta''\nu_1'j' \wedge \beta''\nu' |= |\beta m' \wedge \beta \nu|.$$ 
Again we have $s_{\beta'',\beta}r_{\nu',\nu}\neq 0$ and hence $\mu(C(\beta'' \nu'))=\mu(C(\beta \nu)).$ As a result, we can replace $c_{\beta'' \nu_1',\check{\nu'},j'}$ by $c_{\beta \nu_1, \check{\nu},\check{m'}}.$ Consequently, 
\begin{align}\label{eq:ind_9}
U_{\varphi}(\Sigma_2)&[j'\neq \check{\nu'},m'\neq \nu]
\\\nonumber
&=\sum_{\beta'\neq \beta'',\nu', j'\neq \check{\nu'},i',m'\neq \nu}c_{\beta \nu_1, \check{\nu},\check{m'}}s_{\beta',i'}r_{\nu_1'j',m'}\sigma_{\beta'',\beta}r_{\nu',\nu}(s_{\beta',\beta})^* \otimes f_{i',\beta,m'}\\
\nonumber
&= \sum_{\beta'\neq \beta'',\nu',i',m'\neq \nu}c_{\beta \nu_1, \check{\nu},\check{m'}}s_{\beta',i'}(r_{\nu_1',m_1'}-r_{\nu',m'})\sigma_{\beta'',\beta}r_{\nu',\nu}(s_{\beta',\beta})^* \otimes f_{i',\beta,m'}\\
\nonumber
{}^{(\ast)}&= - \sum_{\beta'\neq \beta'',\nu',i',m'\neq \nu}c_{\beta \nu_1, \check{\nu},\check{m'}}s_{\beta',i'}r_{\nu',m'}\sigma_{\beta'',\beta}r_{\nu',\nu}(s_{\beta',\beta})^* \otimes f_{i',\beta,m'}\\
\nonumber
&=-\sum_{\beta',\nu',i',m'\neq \nu}c_{\beta \nu_1, \check{\nu},\check{m'}}s_{\beta',i'}r_{\nu',m'}(1-\sigma_{\beta',\beta})r_{\nu',\nu}(s_{\beta',\beta})^* \otimes f_{i',\beta,m'}\\
\nonumber
{}^{(\ast \ast)}&=-\sum_{\beta',\nu',i',m'\neq \nu} \mu (C(\nu_{|\nu|})) s_{\beta',i'}r_{\nu',m'}\sigma_{\beta',\beta}r_{\nu',\nu}(s_{\beta',\beta})^* \otimes f_{i',\beta,m'}.
\end{align}
To derive ($\ast$) note that the sum involving $r_{\nu_1',m_1'}$ but not $r_{\nu',m'}$ is zero. Indeed, first observe that $r_{\nu_1',m_1'}$ commutes with $\sigma_{\beta'',\beta}$ as $\ell\geq 2$. Then, $r_{\nu_1',m_1'}$ gets absorbed into $r_{\nu',\nu}$. Then, considering the equality $\sum_{\nu'} r_{\nu',\nu}=1$, brings $\sigma_{\beta'',\beta}$ next to $(s_{\beta',\beta})^*$ which gives zero. Now to derive ($\ast \ast$) observe that the non-zero terms have $r_{\nu',m'}\sigma_{\beta',\beta}r_{\nu',\nu}\neq 0$. In particular, $(s_{\beta' \hat{\nu'},\beta \hat{m'}})^*s_{\beta'\hat{\nu'},\beta \hat{\nu}}\neq 0$. This means that 
$$|\beta \hat{m'} \wedge \beta \hat{\nu}| = |\beta' \hat{\nu'} \wedge \beta' \hat{\nu'}|=\ell,$$ 
so in other words $\hat{m'}=\hat{\nu}$. As $m'\neq \nu$ we obtain $m'_{|m'|}\neq \nu_{|\nu|}.$ Consequently, 
\begin{align*}
c_{\beta \nu_1, \check{\nu},\check{m'}}&=-|\beta \nu_1 \check{m'} \wedge \beta \nu_1 \check{\nu}|_{\lambda}^{-\df} \mu (C(\beta \nu_1 \check{\nu}))= -|\beta m' \wedge \beta \nu|_{\lambda}^{-\df} \mu (C(\beta \nu))=-\lambda_{\max}^{\ell}\mu (C(\beta \nu)) \\
&=- \mu (C(\nu_{|\nu|})).
\end{align*}

Moreover, from any non-zero term in $U_{\varphi}(\Sigma_2)[j'\neq \check{\nu'}, m'=\nu]$, working as before, we obtain $$|\beta''\nu_1'\hat{j'} \wedge \beta''\hat{\nu'}|= |\beta \hat{\nu} \wedge \beta \hat{\nu}|=\ell,$$ and since $|\beta''\nu_1'\hat{j'}|=|\beta''\hat{\nu'}|=\ell$ we get $\beta''\nu_1'\hat{j'}=\beta''\hat{\nu'}.$ Also, $j'\neq \check{\nu'}$ and so
\begin{align*}
c_{\beta'' \nu_1',\check{\nu'},j'}&=-|\beta''\nu_1'j' \wedge \beta''\nu' |_{\lambda}^{-\df} \mu(C(\beta'' \nu'))=-\lambda_{\max}^{\ell}\mu(C(\beta \nu))=-\mu(\sigma_A^{\ell}(C(\beta \nu)))\\
&=-\mu(C(\nu_{|\nu|})).
\end{align*}
With similar computations we obtain
\begin{align}\label{eq:ind_10}
U_{\varphi}(\Sigma_2)[j'\neq \check{\nu'}, m'=\nu]&=\mu(C(\nu_{|\nu|}))\otimes f_{\beta,\beta,\nu}\\ 
\nonumber
&- \sum_{\beta',\nu',i'}\mu(C(\nu_{|\nu|})) s_{\beta',i'}r_{\nu',\nu}\sigma_{\beta',\beta}r_{\nu',\nu}(s_{\beta',\beta})^*\otimes f_{i',\beta,\nu}.
\end{align}

Furthermore, we have
\begin{equation}\label{eq:ind_11}
U_{\varphi}(\Sigma_3)=\sum_{\beta',\nu'\neq j,i',m'}c_{\beta',\nu',j} s_{\beta',i'}r_{j,m'}\sigma_{\beta',\beta}r_{\nu',\nu}(s_{\beta',\beta})^*\otimes f_{i',\beta,m'},
\end{equation}
for which 
\begin{align*}
U_{\varphi}(\Sigma_3)[m'\neq \nu]&=\sum_{m'\neq \nu}c_{\beta,\nu,m'}\otimes f_{\beta,\beta,m'}\\
&+ \sum_{\beta',\nu',i',m'\neq \nu} \mu(C(\nu_{|\nu|})) s_{\beta',i'}r_{\nu',m'}\sigma_{\beta',\beta}r_{\nu',\nu}(s_{\beta',\beta})^*\otimes f_{i',\beta,m'}
\end{align*}
and 
\begin{align*}
U_{\varphi}(\Sigma_3)[m'=\nu]&= - \mu(C(\nu_{|\nu|})) f_{\beta,\beta,\nu}\\
&+ \sum_{\beta',\nu',i'}\mu(C(\nu_{|\nu|})) s_{\beta',i'}r_{\nu',\nu}\sigma_{\beta',\beta}r_{\nu',\nu}(s_{\beta',\beta})^*\otimes f_{i',\beta,\nu}.
\end{align*}

In addition, with similar methods one can show that
\begin{equation}\label{eq:ind_12}
U_{\varphi}(\Sigma_4)=\sum_{\beta',\nu',i',m'}(c_{\beta,\nu,\nu}+1)s_{\beta',i'}r_{\nu',m'}\sigma_{\beta',\beta}r_{\nu',\nu}(s_{\beta',\beta})^*\otimes f_{i',\beta,m'}.
\end{equation}
Thus several cancellations take place in the computation for $U_{\varphi}(1\otimes |D|)U_{\varphi}^*(1\otimes f_{\beta,\beta,\nu}).$ In fact, we have shown that 
\begin{align}\label{eq:ind_13}
U_{\varphi}(1\otimes |D|)U_{\varphi}^*(1\otimes f_{\beta,\beta,\nu})&= (3 + c_{\beta,\nu,\nu}-\lambda_{\max}\mu(C(\beta)\setminus C(\beta \nu_1)))f_{\beta,\beta,\nu} + \sum_{m'\neq \nu}c_{\beta,\nu,m'}\otimes f_{\beta,\beta,m'}\\
\nonumber
&+ (\lambda_{\max}\mu(C(\beta)\setminus C(\beta \nu_1))-2)\sum_{\beta',\nu',i',m'}s_{\beta',i'}r_{\nu',m'}\sigma_{\beta',\beta}r_{\nu',\nu}(s_{\beta',\beta})^*\otimes f_{i',\beta,m'}.
\end{align}
Equating this with \eqref{eq:ind_5} gives that, if $(i',m')\neq (\beta,\nu)$ then 
\begin{equation}\label{eq:ind_14}
\sum_{\beta',\nu'}s_{\beta',i'}r_{\nu',m'}\sigma_{\beta',\beta}r_{\nu',\nu}(s_{\beta',\beta})^*=0,
\end{equation}
since $\lambda_{\max}\mu(C(\beta)\setminus C(\beta \nu_1))-2\neq 0$. The latter fact follows from noticing that
\begin{align*}
\lambda_{\max}\mu(C(\beta)\setminus C(\beta \nu_1))&\leq \lambda_{\max} \mu(C(\beta))=\sum_{j}\lambda_{\max} \mu (C(\beta j))=\sum_j \mu(\sigma_A(C(\beta j)))\\
&\leq \sum_j \mu (C(j))=1.
\end{align*}

We now claim that \eqref{eq:ind_14} implies that for all $\alpha,\beta\in V_A^{1}$ and $\delta,\nu\in V_A^{\ell}$ with $\alpha \delta,\beta \nu \in V_A$, we have
\begin{equation}\label{eq:ind_15}
\sigma_{\alpha,\beta}r_{\delta,\nu}\sigma_{\alpha,\beta}= r_{\delta,\nu}\sigma_{\alpha,\beta}r_{\delta,\nu}\sigma_{\alpha,\beta},
\end{equation}
which in turn, following Lemma \ref{lem:FP_theorem}, means that $$\sigma_{\alpha,\beta}r_{\delta,\nu}=r_{\delta,\nu}\sigma_{\alpha,\beta},$$ and so the $\mathbb G$-action factors through the $\QG_A^{\ell+1}$-action. Indeed, multiplying \eqref{eq:ind_14} by $(s_{\alpha,i'})^*$ (where $i' \in V_A^1$) from the left, and by $s_{\alpha,\beta}$ from the right, we obtain that
\begin{equation}\label{eq:ind_17}
\sigma_{\alpha,i'}\left[\sum_{\nu'}r_{\nu',m'}\sigma_{\alpha,\beta}r_{\nu',\nu}\right] \sigma_{\alpha,\beta}=0.
\end{equation}
In particular, for $m'\neq \nu$ we can sum \eqref{eq:ind_17} over all $i'\in \{1,\ldots,N\}$ and obtain that 
\begin{equation}\label{eq:ind_18}
\sum_{\nu'}r_{\nu',m'}\sigma_{\alpha,\beta}r_{\nu',\nu} \sigma_{\alpha,\beta}=0.
\end{equation}
Multiplying \eqref{eq:ind_18} by $r_{\delta,m'}$ from the left gives 
\begin{equation}\label{eq:ind_19}
r_{\delta,m'}\sigma_{\alpha,\beta}r_{\delta,\nu} \sigma_{\alpha,\beta}=0.
\end{equation}
Then, we have
\begin{align*}
\sigma_{\alpha,\beta}r_{\delta,\nu}\sigma_{\alpha,\beta}&=r_{\delta,\nu}\sigma_{\alpha,\beta}r_{\delta,\nu}\sigma_{\alpha,\beta}+(1-r_{\delta,\nu})\sigma_{\alpha,\beta}r_{\delta,\nu}\sigma_{\alpha,\beta} = r_{\delta,\nu}\sigma_{\alpha,\beta}r_{\delta,\nu}\sigma_{\alpha,\beta}+\sum_{m'\neq \nu} r_{\delta,m'}\sigma_{\alpha,\beta}r_{\delta,\nu} \sigma_{\alpha,\beta}\\
&=r_{\delta,\nu}\sigma_{\alpha,\beta}r_{\delta,\nu}\sigma_{\alpha,\beta}.
\end{align*}
This establishes the inductive step and the proof of Theorem \ref{thm:factoring} is now complete.

\end{proof}

\section{Properties of quantum isometry groups}\label{sec:properties}

In this concluding section\footnote{In this section we will often omit the `comma' from $x_{i,j}$ and simply write $x_{ij}$ when the meaning is evident.}, we will study more closely the compact quantum groups introduced in Definition \ref{defn: Ariadne_0}. Even though they are -- to the best of our knowledge -- new objects, they have connections to important families of examples which have been intensively studied in the quantum group literature, like quantum automorphism groups of graphs and easy quantum groups. We will highlight these connections and show how they can, for instance, be used to study ergodicity of the action of the quantum isometry group on $O_{A}$.

\subsection{Classical versions}
\label{Sec:classical}

Given a compact quantum group $\QG$ its \emph{classical version} $\QGc$ is a compact group such that $C(\QGc):=C(\QG)/J_{c}$, where $J_c$ denotes the commutator ideal of $C(\QG)$, with the coproduct of $C(\QGc)$ induced from that of $C(\QG)$. General considerations show that if $\QG$ is the quantum isometry group of a given structure, then $\QGc$ is the corresponding classical isometry group (see for example \cite{Adam}). Thus we obtain immediately the following proposition.

\begin{prop}
The classical version of the quantum groups $\QG_A^{\ell}$ (for $\ell \in \N \cup \{\infty\}$) is the group $\T \wr \Aut(A)$.	
\end{prop} 

\begin{proof}
	This can be shown directly (noting in particular that in the quotient of $\CA$ by the commutator ideal the matrices $p$ and $q$ of Definition \ref{defn: Ariadne_0} become equal) or deduced from \cite[Theorem 6.3]{gerontogiannis25heat}.
\end{proof}

Note in particular, following Remark \ref{rem: PFvector}, that for the classical version of $\QG_A^{\ell}$, the preservation of the  Perron--Frobenius eigenvector appearing in part (I) of Definition \ref{defn: Ariadne_0} becomes automatic. This follows from the fact that the eigenvalue $\lambda_{\max}$ is simple.

\subsection{The Cuntz case}

Let $\mathbf{1}_{N}$ denote the matrix with all coefficients equal to $1$. If $A = \mathbf{1}_{N}$, then $O_{A}$ is the well-known \emph{Cuntz algebra} \cite{cuntz77simple} usually denoted by $O_{N}$. In that case, the relation $Ap = qA$ follows from the fact that $p$ and $q$ are magic unitaries, so that relation (II) disappears from Definition \ref{defn: Ariadne_0}, and by Remark \ref{rem: PFvector} so does the second part of relation (I). It turns out that the compact quantum groups $\QG_{\mathbf{1}_{N}}^{\ell}$ coincide with some specific unitary easy quantum group first discovered by A.\,Mang in his PhD thesis \cite{Mang2022phd}.

To prove this, we will first give a different description of the quantum groups $\QG_{\textbf{1}_{N}}^{\ell}$, starting from the free complexification $\widetilde{H}_{N}^{+}$ of the quantum hyperoctahedral group $H_{N}^{+}$ (see  \cite{banica08note}):

\begin{definition}
The $C^*$-algebra $C(\widetilde{H}_{N}^{+})$ is the universal $C^*$-algebra generated by elements $(u_{ij})_{1\leqslant i, j\leqslant N}$ such that
\begin{itemize}
\item the matrices $u$ and $\overline{u}$ are unitary;
\item $u_{ij}u_{it}^{*} = 0$ if $j\neq t$ and $u_{ij}u_{sj}^{*} = 0$ if $i\neq s$, for all $1\leqslant i, j, s, t\leqslant N$.
\end{itemize}
The map $\Delta : C(\widetilde{H}_{N}^{+})\to C(\widetilde{H}_{N}^{+})\otimes C(\widetilde{H}_{N}^{+})$ given on the generators by
\begin{equation*}
\Delta(u_{ij}) = \sum_{k=1}^{N}u_{ik}\otimes u_{kj}, \;\;\; i,j=1, \ldots,N,
\end{equation*}
is well-defined and endows $C(\widetilde{H}_{N}^{+})$ with the structure of a compact quantum group, denoted by $\widetilde{H}_{N}^{+}$.
\end{definition}

Let us recall two well-known facts about these quantum groups, with proofs for completeness.

\begin{prop}
The generators of $C(\widetilde{H}_{N}^{+})$ are partial isometries, and their range and source projections form magic unitary matrices.
\end{prop}

\begin{proof}
The second relation can be written in a more compact form as ($1\leqslant i, j, s, t\leqslant N$)
\begin{equation*}
\delta_{is}u_{ij}u_{st}^{*} = \delta_{jt}u_{ij}u_{st}^{*}.
\end{equation*}
Because $u$ is unitary, this implies that
\begin{align*}
u_{ij}^{*}u_{ij} & = u_{ij}^{*}u_{ij}\sum_{t=1}^{N}u_{it}^{*}u_{it} = \sum_{t=1}^{N}u_{ij}^{*}u_{ij}u_{it}^{*}u_{it} = u_{ij}^{*}u_{ij}u_{ij}^{*}u_{ij}.
\end{align*}
In other words, $u_{ij}$ is a partial isometry, which implies in turn that $u_{ij}u_{ij}^{*}$ is also a projection. Now, the unitarity of $u$ and $\overline{u}$ translates into the fact that the source and range projections of the generators form magic unitary matrices.
\end{proof}

We will now realise our quantum isometry groups as quantum subgroups of $\widetilde{H}_{N}^{+}$. Indeed, we can define them as follows.

\begin{definition}
For $\ell\geqslant 2$, we denote by $C(\QG_{\ell}^N)$ the quotient of $C(\widetilde{H}_{N}^{+})$ by the relations
\begin{equation*}
[u_{i_{2}j_{2}}\cdots u_{i_{n}j_{n}}u_{i_{n}j_{n}}^{*}\cdots u_{i_{2}j_{2}}^{*}, u_{i_{1}j_{1}}^{*}u_{i_{1}j_{1}}] = 0,
\end{equation*}
for all $2\leqslant n\leqslant\ell$ and $1\leqslant i_{1}, \ldots, i_{n}, j_1, \ldots,j_n\leqslant N$.
\end{definition}

This is not exactly the definition we had before. Let us therefore check that they match.

\begin{lemma}\label{lem:isomGlN}
The $C^*$-algebra $C(\QG_{\ell}^N)$ is the quotient of $C(\widetilde{H}_{N}^{+})$ by the relations making $u_{i_{1}j_{1}}\cdots u_{i_{n}j_{n}}$ a partial isometry for all $1\leqslant n\leqslant \ell$. In particular, we have $\QG_\ell^N = \QG_{\textbf{1}_{N}}^{\ell}$ for every $\ell \in \N\cup \{\infty\}$.
\end{lemma}

\begin{proof}
We have to show that each set of relations implies the other. Assume first that the commutation relations are satisfied and let us prove by induction that $u_{i_{1}j_{1}}\cdots u_{i_{n}j_{n}}$ is a partial isometry. This is true for $n = 1$. If it holds for some $n$, then $u_{i_{2}j_{2}}\cdots u_{i_{n+1}j_{n+1}}$ is a partial isometry. The commutation relation then means that the range projection of that partial isometry commutes with the source projection of $u_{i_{1}j_{1}}$ which implies that $u_{i_{1}j_{1}}\cdots u_{i_{n+1}j_{n+1}}$ is a partial isometry, thereby proving the property for $n+1$.
	
Conversely, let us assume that $u_{i_{1}j_{1}}\cdots u_{i_{n}j_{n}}$ is a partial isometry for all $1\leqslant n\leqslant \ell$. In particular, $u_{i_{2}j_{2}}\cdots u_{i_{n}j_{n}}$ is a partial isometry, and the fact that multiplying it by $u_{i_{1}j_{1}}$ yields a partial isometry implies that its range projection commutes with $u_{i_{1}j_{1}}^{*}u_{i_{1}j_{1}}$, so that the proof is complete.
\end{proof}

We will now show that $\QG_{\textbf{1}_{N}}^{\ell}$ is the same as a unitary easy quantum group (in the sense of \cite{tarrago17unitary}, but we will not need any element of the theory hereafter) introduced in A.\,Mang's PhD work \cite{Mang2022phd}, a work which has not appeared in article form so far. Following the notations of \cite[Chap 1]{Mang2022phd}, we set for $\ell\geqslant 2$
\begin{equation*}
R_{\ell} = \{\circ^{k}, \bullet^{k} \mid 1\leqslant k\leqslant \ell\}.
\end{equation*}
This is a \emph{parameter set} in the sense of \cite[Chap 1, Def 3.1]{Mang2022phd} and therefore defines a hyperoctahedral unitary easy compact quantum group which will be denoted by $H_{N}^{[\ell]+}$ in the sequel (the $+$ sign is meant to distinguish these from their orthogonal versions $H_{N}^{[\ell]}$ defined by S. Raum and M. Weber in \cite{RaumWeber}). By \cite[Chap 1, Thm 9.4]{Mang2022phd}, these quantum groups are distinct for distinct values of $\ell$. Let us now prove that they coincide with our $\ell$-Ariadne quantum groups.

\begin{prop}
For any $\ell\in \N\cup \{\infty\}$, $\QG_{\textbf{1}_{N}}^{\ell} = H_{N}^{[\ell]+}$.
\end{prop}

\begin{proof}
We shall use the identification of Lemma \ref{lem:isomGlN} and first prove that the generators $(v_{ij})_{1\leqslant i, j\leqslant N}$ of $C(H_{N}^{[\ell]+})$ satisfy the defining relations of the generators $(u_{ij})_{1\leqslant i, j\leqslant N}$ of $C(\QG_{\ell}^N)$. Let us recall the key relations in the definition of $H_{N}^{[\ell]+}$: the fact that the partition $\pi_{\circ^{k}}$ gives an intertwiner is equivalent to the equality (see \cite[Chap 1, Par 7.2]{Mang2022phd})
\begin{equation*}
\delta_{\mathbf{j}, \mathbf{t}}\times v_{i_{1}j_{1}}\cdots v_{i_{k}j_{k}}v_{s_{k}t_{k}}^{*}\cdots v_{s_{1}t_{1}}^{*} = \delta_{\mathbf{i}, \mathbf{s}}\times v_{i_{1}j_{1}}\cdots v_{i_{k}j_{k}}v_{s_{k}t_{k}}^{*}\cdots v_{s_{1}t_{1}}^{*}
\end{equation*}
For $k = 1$, this means that $v_{i_{1}j_{1}}v_{s_{1}t_{1}}^{*}$ vanishes if $i_{1} = s_{1}$ but $j_{1}\neq t_{1}$ or if $i_{1}\neq s_{1}$ but $j_{1} = t_{1}$. In other words, the generators satisfy the defining relations of $C(\widetilde{H}_{N}^{+}) = C(\QG_{1}^N)$. Assume now that we have proven that the relations coming from $\pi_{\circ^{k}}$ imply that any product of $k$ generators is a partial isometry and let us write the relations corresponding to $k+1$. Setting $V = v_{i_{2}j_{2}}\cdots v_{i_{\ell+1}j_{\ell+1}}$, we have
\begin{align*}
v_{i_{1}j_{1}}^{*}v_{i_{1}j_{1}}VV^{*} & = \sum_{s_{1}=1}^{N} v_{i_{1}j_{1}}^{*}v_{i_{1}j_{1}}VV^{*}v_{s_{1}j_{1}}^{*}v_{s_{1}j_{1}} = v_{i_{1}j_{1}}^{*}v_{i_{1}j_{1}}VV^{*}v_{i_{1}j_{1}}^{*}v_{i_{1}j_{1}} \\
& = \sum_{t_{1}=1}^{N}v_{i_{1}t_{1}}^{*}v_{i_{1}t_{1}}VV^{*}v_{i_{1}j_{1}}^{*}v_{i_{1}j_{1}} = VV^{*}v_{i_{1}j_{1}}^{*}v_{i_{1}j_{1}}.
\end{align*}
In other words, the source projection of $v_{i_{1}j_{1}}$ commutes with the range projection of $V$, which implies that their product is a partial isometry. As a consequence, we have a surjective $*$-homomorphism $\Phi : C(\QG_{\ell}^N)\to C(H_{N}^{[\ell]+})$ sending $u_{ij}$ to $v_{ij}$.
	
To prove that $\Phi$ is an isomorphism, it is enough to show that the generators $(u_{ij})_{1\leqslant i, j\leqslant N}$ satisfy the relations associated to $\pi_{\circ^{k}}$ for all $1\leqslant k\leqslant \ell$. But because a product of two partial isometries is a partial isometry if and only if the source projection of the first one commutes with the range projection of the second one, we can simply reverse the computations above, and the result follows.
\end{proof}

Let us record a direct corollary of this, answering the question raised in Remark \ref{rem:distinct} in this particular case, which comes from \cite[Thm 7.1]{Mang2022phd}

\begin{cor}\label{cor:diff_Gl}
The compact quantum groups $\QG_{\textbf{1}_{N}}^{\ell}$ are not isomorphic for distinct values of $\ell$.
\end{cor}

\begin{remark}
The previous description of $\QG_{\mathbf{1}_{N}}^{\ell}$ as a unitary easy quantum group suggests to try to use partition techniques to study their structure and representation theory (see for instance \cite{freslon2023compact}). However, a crucial first step would be to prove that the maps associated to the partitions defining $H_{N}^{[\ell]+}$ are linearly independent, which might prove difficult due to the presence of crossings.
\end{remark}

\subsection{The structure of the quantum isometry group and quantum graph automorphisms}\label{sec:QAut}

Thinking of the matrix $A$ as the adjacency matrix of a directed graph, one can try to relate $\QG_{A}^{\ell}$ to the corresponding \emph{quantum automorphism group} $\QAut(A)$. As mentioned in Subsection \ref{Sec:classical}, the classical isometry group of $O_{A}$ indeed decomposes as a semi-direct product $\mathbb{T}\wr \Aut(A) = \mathbb{T}^{N}\rtimes \Aut(A)$. As we will see, the quantum situation is much more intricate. Let us explain first what we mean by quantum automorphism group, following \cite{banica05quantum}.

\begin{definition}\label{def:QAut}
Let $C(\QAut(A))$ be the universal $C^*$-algebra generated by the coefficients $p_{ij}$, $1\leqslant i, j\leqslant N$ of a magic unitary matrix satisfying the relation
\begin{equation*}
Ap = pA.
\end{equation*}
\end{definition}

Endowed with the unique $*$-homomorphism $$\Delta : C(\QAut(A))\to C(\QAut(A))\otimes C(\QAut(A))$$ such that $\Delta(p_{ij}) = \sum_{k=1}^{N}p_{ik}\otimes p_{kj}$ for all $i,j=1,\ldots, N$, this is a compact quantum group called the \emph{quantum automorphism group} of $A$.

Let us first focus on the torus part $\T^{N}$ in the decomposition of $\T \wr \Aut(A)$, noting that its action on $O_A$ coincides with the standard (multiparameter version of) gauge action. It is natural to consider its quantum counterpart to be the dual of the free group $\F_{N}$.  Recall that a quantum group $\QH$ is said to be a \emph{quantum subgroup} of $\QG$ if there exists a surjective quantum group homomorphism $\pi : \CG\to C(\QH)$.

\begin{prop}\label{prop:quantumgauge}
The compact quantum group $\GAi$ contains $\widehat{\F}_{N}$ as a quantum subgroup, as does $\GA$ for all $\ell\in \N$. 
\end{prop}

\begin{proof}
Simply observe that if $(g_{1}, \ldots, g_{N})$ denote the standard free generators of $\F_{N}$, then setting $v_{ij} := \delta_{ij}g_{i}$, $i,j=1, \ldots, N$ defines a matrix satisfying the defining relations of $\CAi$, providing the desired surjective $*$-homomorphism.

\end{proof}

This looks encouraging, but we run into the following issue: even though $\GAi$ contains the classical graph automorphism group $\Aut(A)$, it is unclear whether $\GAi$ also contains the quantum automorphism group $\QAut(A)$. To get a better grasp at this problem, let us show that the result at least holds when considering the larger quantum group $\QG_{A}^{1}$:

\begin{prop}\label{prop:QAutSubgroup}
The quantum group $\QG_{A}^{1}$ contains $\QAut(A)$ as a quantum subgroup. Moreover, the quantum group $\QG_A^{\infty}$ contains $\Aut(A)$ as a quantum subgroup.
\end{prop}

\begin{proof}
Consider the magic unitary matrix $p=(p_{ij})_{1\leqslant i, j\leqslant N}$ of $C(\QAut(A))$. Then, $p_{ij}p_{ij}^{*} = p_{ij} = p_{ij}^{*}p_{ij}$. Moreover, $Ap = pA$, so that the defining relations of $C(\QG_{A}^{1})$ are satisfied, provinding the first part of the statement. As for the second part, simply observe that $C(\Aut(A))$ is obtained by making all the generators of $C(\QAut(A))$ commute, so that in particular all the commutation relations defining $\CAi$ are satisfied.
\end{proof}

To investigate the problem further, we will now introduce a specific `quantum quotient' of $\GAi$. More precisely, let us denote by $B$ the $C^*$-subalgebra of $\CAi$ generated by the elements $p_{ij}=v_{ij}v_{ij}^*$ and $q_{ij}=v_{ij}^*v_{ij},$ where $v_{ij}$ are the generators of $\CAi$. Because these two matrices are representations, we have $\Phi_{A}(B)\subset B\otimes B$, hence $B = C(\QH_A)$ for some compact quantum group $\QH_{A}$. We will now show that  $\QH_{\bf{1}_N}$ contains the quantum permutation group $\mathbb S_N^+$, and even two of its copies.

\begin{prop}\label{prop:magicinside}
There exists a surjective $*$-homomorphism 
$$\pi  :C(\QH_{\bf{1}_N}) \to C(\mathbb S_N^+\times \mathbb S_N^+),$$ 
mapping $v_{ij}v_{ij}^{*}$ to $P_{ij}$ and $v_{ij}^{*}v_{ij}$ to $Q_{ij}$ for all  $1\leqslant i,j\leqslant N$, where $(P_{ij})_{i,j=1}^{N}$ and $(Q_{ij})_{i, j = 1}^{N}$ are the magic unitaries generating respectively the first and second copies of $C(\mathbb S_N^+)$.
\end{prop}

\begin{proof}
Let $A=\bf{1}_N$ and denote by $B'$ the subalgebra of $B$ generated by the coefficients of $p$ only. Realise $C(\mathbb S_N^+)$ on a Hilbert space $H$, and define $$K= \bigotimes_{k \in \Z} H,$$ viewing $C(\mathbb S_N^+)^{(k)}$ as acting on $K$ (at the $k$-th component). Strictly speaking we need to choose a reference vector to make sense of the infinite tensor product, but here we can simply choose any unit vector $\Omega$ in $H$; this will play no role in the considerations below. Let $\sigma \in B(K)$	denote the right shift and set (for $1\leqslant i, j\leqslant N$) 
\[ w_{ij} = \sigma P_{ij}^{(0)} = P_{ij}^{(1)} \sigma, \]
so that  for $\xi = (\xi_k)_{k \in \Z} \in K$ and $l \in \Z$ we have
\[(w_{ij}(\xi))_l = \begin{cases} \xi_{l-1} & \textup{if }  l\neq 1, \\ P_{ij} \xi_0 & \textup{if }  l = 1.
\end{cases}\]
(Formally one should remember that the formulas as above only make sense if all but finitely many $\xi_l$ are equal to $\Omega$.) Then, $w_{ij}^*w_{ij} = P_{ij}^{(0)}$, $w_{ij} w_{ij}^* = P_{ij}^{(1)}$ for all $i, j=1, \ldots, N$, and analogously for $m \in \N$ and $i_1, \ldots, i_m, j_1, \ldots, j_m \in \{1, \ldots,N\}$
\[ w_{i_1 j_1}\cdots  w_{i_m j_m} = \sigma^m P_{i_1j_1}^{(1-m)} \cdots P_{i_m j_m}^{(0)} = P_{i_1j_1}^{(1)} \cdots P_{i_m j_m}^{(m)} \sigma^m  \]
\[ (w_{i_1 j_1}\cdots  w_{i_m j_m})^* w_{i_1 j_1}\cdots  w_{i_m j_m} = P_{i_1j_1}^{(1-m)}\cdots P_{i_mj_m}^{(0)} = P_{i_mj_m}^{(0)} \cdots  P_{i_1j_1}^{(1-m)},  \] 
\[ w_{i_1 j_1}\cdots  w_{i_m j_m} (w_{i_1 j_1}\cdots  w_{i_m j_m})^* = P_{i_1j_1}^{(1)}\cdots P_{i_mj_m}^{(m)} .\]
Thus we have that the source and range  projections of any $w_{i_1 j_1}\cdots  w_{i_m j_m}$ as above commute (acting on different parts of the tensor product). Also, the relations defining $\QG_{\bf{1}_N}^{\infty}$ are satisfied, so that we obtain a $*$-homomorphism $\rho:C(\QG_{\bf{1}_N}^{\infty}) \to B(K)$ mapping each $v_{ij}$ to the respective $w_{ij}$. By definition we then have $\rho(v_{ij}^*v_{ij}) = P_{ij}^{(0)} \in C(\mathbb S_N^+)$ so that $\pi' = \rho_{\mid B'}$ is a quantum group homomorphism from $B'$ to $C(\mathbb S_N^+)$. If now $B''$ denotes the $C^*$-subalgebra generated by the coefficients of $q$ only, the same argument provides a quantum group homomorphism $\pi'' : B''\to C(\mathbb S_N^+)$ and observing that $\pi'(B')$ and $\pi''(B'')$ commute concludes the proof.
\end{proof}

We can now prove that, contrary to the classical case, the `gauge' part of the quantum isometry group introduced in \ref{prop:quantumgauge} need not be normal.
\begin{cor}
	Let $N \geq 4$. The quantum subgroup $\widehat{\mathbb{F}_N}$ of $\QG_{\mathbf{1}_N}^\infty$ is not normal.
\end{cor}

\begin{proof}
Recall that a compact quantum subgroup $\QH$ of a compact quantum group $\QG$, given by the surjective Hopf *-morphism
$q: \Pol(\QG) \to \Pol(\QH)$, is called \emph{normal} if the respective (algebraic) algebras of left/right $\QH$-invariant functions, $\Pol(\QG/\QH)$ and $\Pol(\QH\setminus \QG)$ coincide, where
\[ \Pol(\QG/\QH) = \{f \in \Pol(\QG): (\id \ot q)(\Delta(f)) = f \otimes 1_\QH\},\]
\[ \Pol(\QH\setminus \QG) = \{f \in \Pol(\QG): (q \ot \id)(\Delta(f)) = 1_\QH \otimes f\}.\]
(see for example \cite{wang14normal}).
Set then $\QG=  \QG_{\mathbf{1}_N}^\infty$, $\QH= \widehat{\mathbb{F}_N}$ and use the notation of the proof of Proposition \ref{prop:quantumgauge}, with $q(v_{ij}) = \delta_{ij} u_i$, $i,j=1, \ldots,N$.
Then it is easy to see that for all $i,j=1, \ldots, N$ we have
\[ (\id \ot q)(\Delta(v_{ij})) =v_{ij} \otimes u_j, \]
\[ (q \ot \id)(\Delta(v_{ij})) =u_i \otimes v_{ij}. \]
Then we have that  
\[  \tilde{B} \subset \Pol(\QG/\QH) \cap  \Pol(\QH\setminus \QG), \]
where $\tilde{B}$, the $^*$-algebra generated by $\{v_{ij}^* v_{ij}, v_{kl}v_{kl}^*:i,j,k,l=1, \ldots,N \}$ is the algebraic version of the algebra introduced before Proposition \ref{prop:magicinside}. The argument below shows in particular that the inclusion is strict.
	
Let then $i,k,l \in \{1, \ldots,N\}$ be pairwise different and consider the element $f=v_{kl} v_{ii}^* v_{ii} v_{ll}^*$. The above formulas show immediately that
	\[ (\id \ot q)(\Delta(f)) = f \otimes u_l (u_i)^* u_i u_l^* = f \otimes 1,\]
	\[  (q\ot \id)(\Delta(f))  = u_k u_i u_i^* u_l^* \otimes f = u_ku_l^*  \otimes f.\]
Thus $f \in \Pol(\QG/\QH)$, and as $u_ku_l^*$ is not a scalar multiple of $1$ to deduce that $f \notin \Pol(\QH\setminus \QG) $ it is enough to show that $f \neq 0$. To this end, consider the $*$-homomorphism $\rho:C(\QG_{\mathbf{1}_N}^{\infty}) \to B(K)$ constructed in the proof of Proposition \ref{prop:magicinside}. It is easy to check that 
	\[ \rho(f) = \sigma (P_{kl} P_{ii} P_{ll})^{(0)} \sigma^*.\]
The latter is however non-zero, as 	$P_{kl} P_{ii} P_{ll}$ is non-zero in $C(\mathbb S_N^+)$ (recall that $N\geq 4$ and use the standard free projections model).	
\end{proof}

Observe an elementary fact: from the relation $Ap = qA$ and the fact that $p$ and $q$ commute, which can be expressed as $(p\otimes q)\Sigma = \Sigma(q\otimes p)$ with the flip map $\Sigma$ on $\mathbb{C}^{N}\otimes \mathbb{C}^{N}$, we conclude that the operator
\begin{equation*}
T_{A} = (A^{t}\otimes A)\Sigma	
\end{equation*}
intertwines $p\otimes q$ with itself. In other words, the $C^*$-algebra $B$ defined before Proposition \ref{prop:magicinside} is a quotient of $C(\QAut(T_{A}))$:
\begin{equation*}
C(\QAut(T_{A})) \twoheadrightarrow B.
\end{equation*}
This leads to the following key observation: if $\QAut(T_{A})$ is trivial, then $B = \mathbb{C} 1$, meaning that $v_{ij}^{*}v_{ij} = \delta_{ij}1$, $i,j=1, \ldots,N$. In other words, the off-diagonal coefficients vanish and diagonal coefficients are unitary. This means that the surjection $\pi : \CAi\to C(\widehat{\F}_{N})$ is injective. 

This implication cannot be reversed in general.

\begin{prop}\label{prop:cex_generation}
There exists a primitive matrix $A\in M_2$ such that $\GAi = \widehat{\F}_{2}$ but $\QAut(T_{A})$ is non-trivial.
\end{prop}

\begin{proof}
Let us set
\begin{equation*}
A = \left(\begin{array}{cc}
1 & 1  \\
1 & 0 \\	
\end{array}\right)
\text{ so that }
T_{A} = \left(\begin{array}{cccc}
1 & 1 & 1 & 1 \\
1 & 1 & 0 & 0 \\
1 & 0 & 1 & 0 \\
1 & 0 & 0 & 0	
\end{array}\right).
\end{equation*}
It is easy to check that $\QAut(A)=\Aut(A)$ is trivial. The relation $Ap=qA$ reads (remembering that sums over rows and columns of a magic unitary equal the unit)
\begin{equation*}
\left(\begin{array}{cc}
1 & 1 \\
p_{11} & p_{12}	
\end{array}\right)=
\left(\begin{array}{cc}
1 & q_{11} \\
1 & q_{21}	
\end{array}\right)
.
\end{equation*}
It follows that $p_{11} = 1$, implying $p_{12} = p_{21} = 0$ and $p_{22} = 2$, and that $q_{11} = 1$, yielding also $q_{12} = q_{21} = 0$ and $q_{22} = 1$. In other words, $\GAi = \widehat{\F}_{2}$. Nevertheless, a direct computation shows that the permutation $(1, 3, 2, 4)$ commutes with $T_{A}$, so that $\QAut(T_{A})$ is non-trivial.
\end{proof}

We do not know yet an example of a matrix $A$ for which $\QAut(A)$ is trivial and yet $\GAi \neq \widehat{\F}_{N}$. This is related to the question below, where `generated' means \emph{topologically generated} in the sense of \cite{brannan17connes}. We ask it for $\QG_A^1$ in view of Proposition \ref{prop:QAutSubgroup}.

\begin{quest}
Is $\QG_A^1$ always generated by $\widehat{\F}_{N}$ and $\QAut(A)$?
\end{quest}



\subsection{Ergodicity}\label{sec:Erg}

Classically, the action of the isometry group on $O_{A}$ can never be ergodic (see below for a proof). Nevertheless, since as we have seen the quantum isometry group behaves differently from the classical one, it makes sense to investigate the potential ergodicity of the action. Recall that given an action $\varphi: B \to C(\QG) \otimes B$ of a compact quantum group on a $C^*$-algebra $B$, we denote
\[ \Fix\,\varphi=\{ a \in B: \varphi (a) = 1 \otimes a\}\]
and say that $\varphi$ is \emph{ergodic} if $\Fix\, \varphi = \mathbb{C} 1$. Suppose that $\QH$ is a quantum subgroup of $\QG$, given by a surjective map $q: \Pol(\QG) \to \Pol(\QH)$, and that $\varphi$ is an action as above; then the action $\varphi$  descends to an action $\varphi_\QH$ of $\QH$ essentially given by the formula $\varphi_\QH= (q \otimes \id)\varphi$ (see for instance \cite{commer17actions} for details concerning actions of compact quantum groups). This is useful since it induces an inclusion
\[ \Fix \, \varphi_\QH \supset \Fix\, \varphi.\]
In particular if $\varphi_\QH$ is ergodic, then so is $\varphi$. We already mentioned that the action of the classical isometry group is not ergodic, let us prove it for completeness (noting also that it can be deduced from the main result of \cite{HKLS}).

\begin{prop}\label{prop:no_erg}
The action of $\mathbb{T}\wr \mathrm{Aut}(A)$	on $O_{A}$ is never ergodic.
\end{prop}

\begin{proof}
Because $A$ is primitive, there is $k\in \N$ such that all the entries of $A^{k}$ are strictly positive. In particular, there exist cycles of length $k$ in $A$, and we may also assume that not all paths in $V_A^k$ are cycles. Let $C_{k}$ denote the set of cycles of length $k$, and set
\begin{equation*}
x_{k} = \sum_{\alpha\in C_{k}}S_{\alpha}S_{\alpha}^{*}.
\end{equation*}
It is then easy to see that $x_k$ is a projection fixed by $\mathbb{T}\wr \mathrm{Aut}(A)$, which cannot be trivial, as seen from the isomorphism described in the last paragraph of Section \ref{sec:preliminaries}.
\end{proof}

Turning to the quantum case, we first observe that ergodicity can be ``checked on a commutative subalgebra'' in the  sense described in the next proposition. We will use without any further comments the isomorphism from the last paragraph of Section \ref{sec:preliminaries}, writing simply $C(\Sigma_A)\subset O_A$.

\begin{prop}\label{prop:commfix}
	For $\varphi:O_A \to C(\QG_A^\infty) \otimes O_A$ we have $\Fix \, \varphi \subset C(\Sigma_A)$. 
\end{prop}

\begin{proof}
The argument is as follows: the corresponding action is determined by the formula 
\[ \varphi_{\widehat{\F}_{N}}(S_\mu S_\nu^*) = u_\mu u_\nu^*\otimes S_\mu S_\nu^* , \; \; \mu, \nu \in V_A,\]
where for the word $\mu=\mu_1 \cdots \mu_n \in V_A$ we write $u_\mu := u_{\mu_1} \cdots u_{\mu_n} \in \mathbb{F}_N \subset C(\widehat{\mathbb{F}_N})$. Consider the canonical conditional expectation $\mathbb{E}: O_A \to \Fix\, \varphi_{\widehat{\F}_{N}}$, given by the formula $\mathbb{E}= (h_{\widehat{\F}_{N}}\otimes \id)\circ \varphi_{\widehat{\F}_{N}}$. The Haar state of $\widehat{\mathbb{F}}_N$ is given simply by the standard trace on the group algebra, so that for  $g \in \mathbb{F}_N$ we have $h_{\widehat{\F}_{N}}(g) = \delta_{g,e}$. Putting these facts together shows that 
\[\mathbb{E} (S_\mu S_\nu^*) = \begin{cases} S_\mu S_\mu^* & \textup{ if } \mu =\nu, \\ 0 & \textup{otherwise.}\end{cases} \]
Thus by continuity and the discussion above the proposition we have 
\[ \Fix \, \varphi \subset \Fix \, \varphi_{\widehat{\F}_{N}} = \overline{\textup{Lin}} \{S_\mu S_\mu^*: \mu \in V_A\} = C(\Sigma_A).  \qedhere\]
\end{proof}

We will use this to prove that ergodicity does hold under some extra assumptions. To do this, we first need to understand further the action on $X_{A}$, thanks to the next proposition.

\begin{prop}\label{prop:comminv}
For $\varphi:O_A \to C(\QG_A^\infty) \otimes O_A$ we have $\varphi ( C(\Sigma_A))  \subset C(\QG_A^\infty) \otimes C(\Sigma_A) $. In other words, the action of  $\QG_A^\infty$ on $O_A$ leaves $C(\Sigma_A)$ invariant (whilst the corresponding action of 
$\QG_A^1$ in general does not). Moreover, if we denote (for $k \in \N$) by $C(\Sigma_A^k)$ the algebra of functions 
on $\Sigma_A$ which are constant on all cylinder sets $Z_\mu$ with $\mu \in V_A, |\mu|=k$, then we also have
$\varphi ( C(\Sigma_A^k))  \subset C(\QG_A^\infty) \otimes C(\Sigma_A^k) $. 
\end{prop}

\begin{proof}
We shall show that for every $ k \in \N$ and $\mu \in V_A^k$ we have 
\begin{equation} \label{alphaSmu} \varphi(S_\mu S_\mu^*) = 
\sum_{\nu \in V_A^k}  w_{\mu, \nu} \otimes S_\nu S_\nu^*,\end{equation}
where $w_{\mu, \nu}:=v_{\mu_1, \nu_1}\cdots v_{\mu_k, \nu_k} v_{\mu_k, \nu_k}^* \cdots v_{\mu_1, \nu_1}^* \in\Pol(\QG_A^\infty)$ (for every $\nu \in V_A^k$).
As $C(\Sigma_A)= C^*\{S_\mu S_\mu^*: \mu \in V_A\}$, and $C(\Sigma_A^k)= C^*\{S_\mu S_\mu^*: \mu \in V_A^k\}$ ($k \in \N$), this will establish the proposition.
	
The proof of \eqref{alphaSmu} will proceed by induction on the length of $\mu$. If $i \in \{1, \cdots,N\}$, then
\[ \varphi(S_i S_i^*) = \sum_{j,k=1}^N v_{i,j} v_{i,k}^* \otimes S_j S_k^* = \sum_{j=1}^N v_{i,j} v_{i,j}^* \otimes S_j S_j^*,\]  
by Remark \ref{rem: orthogonality}, so that the claim holds for $|\mu|=1$. Suppose then that $k \in \N$ and \eqref{alphaSmu} holds for all $\mu \in V_A$ of length $k$. Let $\mu'\in V_A$, $|\mu'|=k+1$, and write $\mu' = i \mu$ for certain $i \in \{1, \cdots,N\}$, $\mu \in V_A^k$. We then have, by the induction assumption,
\begin{align*}
\varphi (S_{\mu'}S_{\mu'}^*) & = \varphi(S_i) \varphi (S_\mu S_\mu^*)\varphi(S_i^*) =
\sum_{j,l=1}^N \sum_{\nu \in V_A^k} v_{i,j} w_{\mu,\nu} v_{i,l}^* \otimes S_j S_\nu S_\nu^* S_l^*. 	
\end{align*}
But then for $\nu, \mu,j,l$ as above, when we note that $w_{\mu, \nu}$ is the range projection of the partial isometry 
$v_{\mu_1, \nu_1}\cdots v_{\mu_k, \nu_k}$, we see that
\[ v_{i,j} w_{\mu,\nu} v_{i,l}^* = v_{i,j} v_{i,j}^* v_{i,j} w_{\mu,\nu} v_{i,l}^* =  v_{i,j}  w_{\mu,\nu}  v_{i,j}^* v_{i,j} v_{i,l}^*,\]
and the latter is equal to 0 whenever $j \neq l$. Thus 
\begin{align*}
\varphi (S_{\mu'}S_{\mu'}^*) & =
\sum_{j=1}^N \sum_{\nu \in V_A^k} v_{i,j} w_{\mu,\nu} v_{i,j}^* \otimes S_j S_\nu S_\nu^* S_j^* = 
\sum_{\nu' \in V_A, |\nu'|=k+1} w_{\mu',\nu'} \otimes S_{\nu'} S_{\nu'}^*, 	
\end{align*}
where we used the fact that $v_{i,j} v_{\mu_1, \nu_1} =0$ unless $j\nu \in V_A$.
\end{proof}

With this at hand, we can now establish the ergodicity of the action of $\QG_{\bf{1}_N}^\infty$ on the Cuntz algebra $O_N$.

\begin{thm}\label{thm:Erg}
The action of $\QG_{\bf{1}_N}^\infty$ on $O_N$ is ergodic.
\end{thm}

\begin{proof}
If the action were not ergodic, the general theory of actions of compact quantum groups (see \cite{commer17actions}) would give us a canonical dense $*$-subalgebra of $C(\Sigma_{\bf{1}_N})$ on which the canonical Hopf $*$-algebra acts, and which contains an element $x$ which is fixed and not a multiple of the unit. It is clear by the definition that this canonical $*$-subalgebra is the one generated by the $*$-subalgebras $C(\Sigma_{\bf{1}_N}^{k})$ for $k\in \N$. Therefore, by virtue of Propositions \ref{prop:commfix} and \ref{prop:comminv}, it is enough to prove the following: for any $k \in\N$ the action $\varphi_k:C(\Sigma_{\bf{1}_N}^k) \to C(\QG_{\bf{1}_N}^\infty) \otimes C(\Sigma_{\bf{1}_N}^k)$ is ergodic. 

Note that the algebra $C(\Sigma_{\bf{1}_N}^k)$ is finite-dimensional and commutative, with minimal projections $p_\mu:=S_\mu S_\mu^*$ with $\mu \in V_{\bf{1}_N}^k$. As $\Fix \, \varphi_k$ is a $*$-subalgebra, it suffices to prove that if $x \in \Fix \, \varphi_k$ is non-negative, then it must be scalar. Consider then $x = \sum_{\mu \in V_{\bf{1}_N}^k} c_\mu p_\mu$, with some scalars $c_\mu \geq 0$. If $\varphi_k(x) = 1 \otimes x$ then, following the adapted notation, \eqref{alphaSmu} implies
\[ \sum_{\mu, \nu \in V_{\bf{1}_N}^k} c_\mu w_{\mu, \nu} \otimes p_\nu = \sum_{\mu \in V_{\bf{1}_N}^k} 1 \otimes c_\mu p_\mu.\]
As the projections $\{p_\mu: \mu \in V_{\bf{1}_N}^k\}$ are linearly independent, the displayed formula is equivalent to the equalities
\[ \sum_{\mu  \in V_{\bf{1}_N}^k} c_\mu w_{\mu, \nu}  = c_\nu 1, \;\;\; \nu \in V_{\bf{1}_N}^k.\]

We need now some elementary observations.
\begin{itemize}
\item The construction in the proof of Proposition \ref{prop:magicinside} shows that $v_{\mu_1, \nu_1}\cdots v_{\mu_k, \nu_k}\neq 0$ as soon as $\mu, \nu\in V_{\bf{1}_N}^{k}$. Therefore, $w_{\mu, \nu}\neq 0$ for all the terms in the sum.
\item The coefficients $c_{\mu}$ which are non-zero are all equal. Indeed, let $c_{\mu_{0}}$ be the smallest non-zero one. Then, for any $\nu\in V_{\bf{1}_N}^{k}$, $c_{\nu}w_{\nu, \mu_0}\leqslant c_{\mu_0}1$, which upon multiplying by $w_{\nu, \mu_0}$ yields $c_{\nu}\leqslant c_{\mu_0}$. By minimality we then have either $c_{\nu} = 0$ or $c_{\nu} = c_{\mu_0}$.
\end{itemize}

Summing up, $x$ must be (up to rescaling) of the form
\[x = \sum_{\mu \in F}  p_\mu, \]
where $F$ is a subset of $V_{\bf{1}_N}^k$. We may assume that $F$ is a non-empty proper subset of $V_{\bf{1}_N}^k$, as otherwise $x$ is scalar. It remains to prove that this leads to a contradiction. Note that now $\varphi(x)=1 \otimes x$ means that
\[ \sum_{\mu \in F} w_{\mu, \nu} = 1\]
for every $\nu \in F$. We know that $\sum_{\mu \in V_{\bf{1}_N}^k} w_{\mu, \nu} =1$, and we already know that given $\mu, \nu \in V_{\bf{1}_N}^k$ we always have $w_{\mu, \nu}\neq0$. 
This ends the proof.
\end{proof}

Let us add that ergodicity does not always hold for the quantum isometry group. Indeed, as shown in Proposition \ref{prop:cex_generation}, it might be the case that $\GAi = \widehat{\mathbb{F}_{N}}$, in which case $\Fix\, \varphi = C(\Sigma_{A})$ is non-trivial. The following question however still remains relevant.

\begin{quest}
	Suppose that the matrix $A$ is such that if $(P_{ij})_{1\leq i,j\leq N}$ is the defining magic unitary of $C(\QAut(A))$ then for every $i,j=1,\ldots,N$ we have $P_{ij}\neq 0$ (in other words, the graph associated to $A$ is \emph{quantum transitive}). Must the action of the quantum isometry group $\QG_A^\infty$ on $O_A$ be ergodic? Perhaps under some additional natural assumptions?
\end{quest} 

Note that to establish this result it would essentially suffice to produce a model of $\QG_A^\infty$ analogous to that constructed for $\QG_{\bf{1}_N}^\infty$ in Proposition \ref{prop:magicinside}.

\smallskip

We would like to conclude this work with a result concerning the quantum group action on $\Sigma_{\bf{1}_N}$, well-known to be a Cantor set. 

\begin{thm} \label{thm:ergfaith}
The restricted map $\varphi:C(\Sigma_{\bf{1}_N})\to C(\QG_{\bf{1}_N}^\infty)\otimes C(\Sigma_{\bf{1}_N})$ defines an ergodic action of the compact matrix quantum group $\QG_{\bf{1}_N}^\infty$ on $C(\Sigma_{\bf{1}_N})$.
\end{thm}
\begin{proof}
The result follows from the proofs of Propositions \ref{prop:commfix}, \ref{prop:comminv} 
and Theorem \ref{thm:Erg}.
\end{proof}

To the best of our knowledge, no example was known of a genuine quantum ergodic action on the Cantor set involving only a compact matrix quantum group (one might build examples through inductive limits of iterated free wreath products following \cite{bhowmickgoswamiskalski}, \cite{bassiconti} and \cite{bronwlow25self}, but these will be `infinitely generated').


\begin{Acknowledgements} The authors would like to thank Magnus Goffeng, Bram Mesland and Christian Voigt for valuable discussions on the subject of this paper.  AS was partially supported by the National Science Center Grant OPUS-29 UMO-2025/57/B/ST1/00057.
	
	We would like to thank the referees for several thoughtful comments and in particular for pointing out the reference \cite{Park95Isometries}.
\end{Acknowledgements}





\end{document}